\documentclass[11pt, twoside]{article}
\usepackage{bcc-cambridge-class}

\DeclareMathOperator{\supp}{supp}
\DeclareMathOperator{\codim}{codim}
\DeclareMathOperator{\bias}{bias}
\DeclareMathOperator{\arank}{arank}
\DeclareMathOperator{\rank}{rank}
\DeclareMathOperator{\prank}{prank}

\bcctitle{Finite field models in arithmetic combinatorics -- twenty years on}
\bccshorttitle{Finite field models in arithmetic combinatorics -- twenty years on}

\bccname{Sarah Peluse}
\bccshortname{Sarah Peluse} 

\bccaddressa{Department of Mathematics \\ University of Michigan \\ East Hall, 530 Church
  Street \\ Ann Arbor, Michigan 48109 \\ USA }
\bccemaila{speluse@umich.edu}

\newtheorem{problem}[example]{Problem}

\begin{document}

\makebcctitle


\begin{abstract}
  About twenty years ago, Green wrote a survey article on the utility of looking at toy
  versions over finite fields of problems in additive combinatorics. This article was
  extremely influential, and the rapid development of additive combinatorics necessitated
  a follow-up survey ten years later, which was written by Wolf. Since the publication of
  Wolf's article, an immense amount of progress has been made on several central open
  problems in additive combinatorics in both the finite field model and integer
  settings. This survey, written to accompany my talk at the 2024 British Combinatorial
  Conference, covers some of the most significant results of the past ten years and
  suggests future directions.
\end{abstract}

The central problems in additive combinatorics typically concern subsets of the
integers. For example, one of Erd\H{o}s's most well-known conjectures, which dates back to
the 1940s, is that any subset $S$ of the natural numbers for which
$\sum_{n\in S}\frac{1}{n}=\infty$ must contain arbitrarily long nontrivial arithmetic
progressions. These foundational problems are all famously hard, and a nontrivial portion
of that difficulty comes from the fact that the integers are not very convenient to work
with--one disadvantage being that they have no nontrivial finite subgroups. Thus, in a
survey article for the 2005 British Combinatorial Conference written almost twenty years
ago, Green~\cite{Green05} suggested that a possible line of attack for some of these
problems is to consider their analogue in high-dimensional vector spaces over finite
fields.

The conjecture of Erd\H{o}s mentioned above can be shown to be essentially equivalent to
the statement that, for each natural number $k$, there exists a constant $c_k>0$ such that
any subset $A\subset \{1,\dots,N\}$ containing no nontrivial $k$-term arithmetic
progressions, i.e., configurations of the form
\begin{equation*}
  x,x+y,\dots,x+(k-1)y
\end{equation*}
with $y\neq 0$, must satisfy a bound of the form
$|A|=O\left(\frac{N}{(\log{N})^{1+c_k}}\right)$. Thus, this conjecture can be addressed by
answering the following question.
\begin{qn}\label{qn:kAP}
  How large can a subset of the first $N$ integers be if it contains no nontrivial $k$-term arithmetic
  progressions?
\end{qn}
This question also makes sense with $\{1,\dots,N\}$ replaced by $\mathbf{F}_p^n$, where we
are interested in the regime with the prime $p\geq k$ fixed (so that the terms of a nontrivial
arithmetic progression are distinct) and $n$ tending to infinity.
\begin{qn}\label{qn:kAPff}
  Fix a prime $p\geq k$. How large can a subset of $\mathbf{F}_p^n$ be if it contains no
  nontrivial $k$-term arithmetic progressions?
\end{qn}
Most other problems in additive combinatorics solely involving addition or subtraction can
be similarly formulated in $\mathbf{F}_p^n$, usually provided that the characteristic $p$
is sufficiently large. For problems that also involve multiplication, some of which we
will discuss later on in the survey, it can be similarly fruitful to consider analogues
phrased in $\mathbf{F}_p$ in the regime where $p$ tends to infinity, in $\mathbf{F}_{p^n}$
in the regime where $p$ is fixed and $n$ tends to infinity, or in the function field
setting $\mathbf{F}_q[t]$.

Two very useful advantages of working in high-dimensional vector spaces versus working in
$\{1,\dots,N\}$ are that one has access to the tools of linear algebra in the former
setting and high-dimensional vector spaces posses a huge number of subgroups. This means
that many technical challenges that arise when working in the integer setting are trivial,
or at least much easier to overcome, in vector spaces over finite fields. There are a few
ways in which $\mathbf{F}_p^n$ is more complicated to work in than $\{1,\dots,N\}$, since
the number of generators of $\mathbf{F}_p^n$ grows as $n$ grows. Overall, though, problems in
additive combinatorics formulated in high-dimensional vector spaces over finite fields tend to be easier
than the same problem formulated in the integers. Despite this, the
$\mathbf{F}_p^n$-analogues of problems preserve most of the essential features of the
problems they are based on, and
there are several tools for translating sufficiently robust proofs from the high
dimensional vector space setting to the integer setting. Similarly, the various finite
field model analogues of problems involving multiplication tend to be easier than the same
problem formulated in the integers for a variety of reasons (e.g., $\mathbf{F}_p$ is
closed under addition and multiplication, while $\{1,\dots,N\}$ is not), but still
preserve the problem's essential features.

Even without the motivation of making progress on some famous problem concerning sets of
integers, finite field analogues of problems are interesting in their own right, often
have applications in theoretical computer science, and have led to the development of some
beautiful mathematics.

Since Wolf's follow-up survey~\cite{Wolf15}, there has been spectacular progress in the
case $k=3$ of Questions~\ref{qn:kAP} and~\ref{qn:kAPff}, and this survey will begin by
covering these recent developments and some related problems, with a focus on the finite
field model setting. We will then discuss progress on the inverse theorems for the Gowers
uniformity norms in the setting of high-dimensional vector spaces over finite fields and
the very recent resolution of the polynomial Freiman--Ruzsa conjecture, and then turn to
the problem of proving a quantitative version of the polynomial Szemer\'edi theorem, where
attacking the problem first in $\mathbf{F}_p$ led to the key developments needed to make
progress in the integer setting. The last section will briefly cover a collection of other
interesting topics: relations between different notions of rank of multilinear forms, the
multidimensional Szemer\'edi theorem, and the notion of true complexity.

\subsection*{Notation and conventions}

Throughout this survey, we will use both the standard asymptotic notation $O,\Omega,$ and
$o$ and Vinogradov's notation $\ll,\gg,$ and $\asymp$. So, for any two quantities $X$ and
$Y$, $X=O(Y)$, $Y=\Omega(X)$, $X\ll Y$, and $Y\gg X$ all mean that $|X|\leq C|Y|$ for some
absolute constant $C>0$, and the relation $X\asymp Y$ means that $X\ll Y\ll X$. We will
also write $O(Y)$ to represent a quantity that is $\ll Y$ and $\Omega(X)$ to represent a
quantity that is $\gg X$. The notation $\log_{(m)}$ means the $m$-fold iterated logarithm and $\exp^{(m)}$ means the $m$-fold iterated exponential.

For any subsets $A$ and $B$ of an abelian group $(G,+)$, we can form their
\defword{sumset} and \defword{product set},
\begin{equation*}
  A+B:=\{a+b:a\in A\text{ and }b\in B\}
\end{equation*}
and
\begin{equation*}
  A-B:=\{a-b:a\in A\text{ and }b\in B\}
\end{equation*}
respectively. We will denote the $k$-fold sumset of a set $A$ with itself by
\begin{equation*}
  kA:=\{a_1+\dots+a_k:a_1,\dots,a_k\in A\}
\end{equation*}
for every natural number $k$. Note that this is not the dilation of the set $A$ by $k$,
which we will instead denote by $k\cdot A$, so that
\begin{equation*}
  k\cdot A:=\{ka:a\in A\}.
\end{equation*}

We will denote the indicator function of a set $S$ by $1_S$, and set $e(z):=e^{2\pi i z}$
for all $z\in\mathbf{R}$ and $e_p(z):=e(z/p)$ for all $z\in\mathbf{F}_p$ and primes $p$. A
complex-valued function is said to be \textit{$1$-bounded} if its modulus is bounded by
$1$, so that indicator functions and complex exponentials $e(\cdot)$ and $e_p(\cdot)$ are
always $1$-bounded.

For any finite, nonempty set $S$ and any complex-valued function $f$ on $S$, we denote the
average of $f$ over $S$ by $\mathbf{E}_{x\in S}f(x):=\frac{1}{|S|}\sum_{x\in S}f(x)$. The Fourier transform of $f:\mathbf{F}_p^n\to\mathbf{C}$ at the frequency
$\xi\in\mathbf{F}_p^n$ is defined to be
\begin{equation*}
  \widehat{f}(\xi):=\mathbf{E}_{x\in\mathbf{F}_p^n}f(x)e_p(-\xi\cdot x).
\end{equation*}
We then have the Fourier inversion formula,
\begin{equation*}
  f(x)=\sum_{\xi\in\mathbf{F}_p^n}\widehat{f}(\xi)e_p(\xi\cdot x)
\end{equation*}
and, for any other function $g:\mathbf{F}_p^n\to\mathbf{C}$, Parseval's identity
\begin{equation*}
  \mathbf{E}_{x\in\mathbf{F}_p^n}f(x)\overline{g(x)}=\sum_{\xi\in\mathbf{F}_p^n}\widehat{f}(\xi)\overline{\widehat{g}(\xi)}.
\end{equation*}

Finally, we will use the standard notation $[N]:=\{1,\dots,N\}$ for the set of the first
$N$ integers.

\section{Roth's theorem, the cap set problem, and the polynomial method}

The first nontrivial case of Question~\ref{qn:kAP} is the case $k=3$, which has attracted
by far the most attention and is one of the central problems in additive combinatorics.
\begin{qn}\label{qn:3AP}
  How large can a subset of the first $N$ integers be if it contains no nontrivial
  three-term arithmetic progressions?
\end{qn}
Roth~\cite{Roth53} was the first to address this question, and also the first to show
that if $A\subset[N]$ contains no nontrivial three-term arithmetic progressions, then
$|A|=o(N)$. Using a Fourier-analytic argument, Roth proved the explicit bound
$|A|\ll\frac{N}{\log\log{N}}$ for the size of such sets.

A subset of $\mathbf{F}_3^n$ having no nontrivial three-term arithmetic progressions is
called a \defword{cap set}, and the $k=3$ case of Question~\ref{qn:kAPff} is known as the
\defword{cap set problem}.
\begin{qn}\label{qn:3APff}
  How large can a subset of $\mathbf{F}_3^n$ be if it contains no nontrivial three-term
  arithmetic progressions?
\end{qn}
Brown and Buhler~\cite{BrownBuhler82} were the first to prove that if
$A\subset\mathbf{F}_3^n$ is a cap set, then $|A|=o(3^n)$. The first quantitative bound in
the cap set problem was proven by Meshulam~\cite{Meshulam95}, who adapted Roth's proof to
the setting of high-dimensional vector spaces over finite fields to show that if $A$ is a cap set, then
$|A|\ll\frac{3^n}{n}$. Note that the savings over the trivial bound $|\mathbf{F}_3^n|$ of
$\log{|\mathbf{F}_3^n|}$ in Meshulam's theorem is exponentially stronger than the savings
of $\log\log{N}$ in Roth's theorem--this turns out to be a direct consequence of the
abundance of subspaces in $\mathbf{F}_3^n$.

Until very recently, all of the work on improving Roth's theorem focused on refining
Roth's original Fourier-analytic proof to make it as efficient as in the finite field
model setting. Work by Heath-Brown~\cite{HeathBrown87} and Szemer\'edi~\cite{Szemeredi90},
Bourgain~\cite{Bourgain99,Bourgain08}, Sanders~\cite{Sanders11,Sanders12},
Bloom~\cite{Bloom16}, and Schoen~\cite{Schoen21} improved the bound in Roth's theorem
right up to the $O\left(\frac{N}{\log{N}}\right)$-barrier by following this approach, with
Schoen's result giving a bound of $|A|\ll\frac{N(\log\log{N})^{3+o(1)}}{\log{N}}$. By a
very careful analysis of the set of large Fourier coefficients of cap sets, Bateman and
Katz~\cite{BatemanKatz12} proved that if $A\subset\mathbf{F}_3^n$ is a cap set, then
$|A|\ll\frac{3^n}{n^{1+c}}$ for some absolute constant $c>0$, thus breaking the
logarithmic barrier in the cap set problem for the first time. Overcoming numerous
difficult technical obstacles, Bloom and Sisask~\cite{BloomSisask20} managed to adapt the
(already very complicated) argument of Bateman and Katz to the integer setting.
\begin{theorem}\label{thm:BloomSisask}
  If $A\subset[N]$ contains no nontrivial three-term arithmetic progressions, then
  \[
    |A|\ll\frac{N}{(\log{N})^{1+c}},
  \]
  where $c>0$ is some absolute constant.
\end{theorem}
This proves the first nontrivial case of the famous conjecture of Erd\H{o}s mentioned in
the introduction. In comparison, in 1946, Behrend~\cite{Behrend46} constructed a subset of
$[N]$ lacking three-term arithmetic progressions of density $\gg\exp(-C\sqrt{\log{N}})$,
where $C>0$ is an absolute constant. This is still, essentially, the best known lower
bound construction.

Earlier this year, there was another spectacular breakthrough in the study of sets lacking
three-term arithmetic progressions. Kelley and Meka~\cite{KelleyMeka23} proved almost
optimal bounds in Roth's theorem using a beautiful argument that almost completely avoids
the use of Fourier analysis and is significantly simpler than the arguments of
Bateman--Katz and Bloom--Sisask, though, like all improvements to the best known bound in
Roth's theorem since the work of Sanders, it uses a result of Croot and
Sisask~\cite{CrootSisask2010} on the almost-periodicity of convolutions.
\begin{theorem}\label{thm:KM}
  If $A\subset[N]$ contains no nontrivial three-term arithmetic progressions, then
  \[
    |A|\ll\frac{N}{\exp(C(\log{N})^{1/12})},
  \]
  where $C>0$ is some absolute constant.  
\end{theorem}
Bloom and Sisask~\cite{BloomSisask2023} have since shown that the exponent $\frac{1}{12}$
of $\log{N}$ appearing in the Kelley--Meka theorem can be improved to $\frac{1}{9}$ by
optimizing the argument in~\cite{KelleyMeka23}.

In contrast to the integer setting, all known constructions of cap sets are exponentially
small, and thus Frankl, Graham, and R\"odl~\cite{FranklGrahamRodl87}, and, independently,
Alon and Dubiner~\cite{AlonDubiner93}, asked whether there exists a positive constant
$c<3$ such that $|A|\ll c^n$ whenever $A\subset\mathbf{F}_3^n$ is a cap set. In
breakthrough work, Croot, Lev, and Pach~\cite{CrootLevPach17} introduced a new variant of
the polynomial method, which they used to show that any subset of
$(\mathbf{Z}/4\mathbf{Z})^n$ containing no nontrivial three-term arithmetic progressions
must have size $\ll 3.61^n$. Ellenberg and Gijswijt~\cite{EllenbergGijswijt17} then used
the Croot--Lev--Pach polynomial method to obtain a similar power-saving bound in the cap set
problem.
\begin{theorem}\label{thm:EG}
  If $A\subset\mathbf{F}_3^n$ is a cap set, then
  \begin{equation*}
    |A|\ll 2.756^n.
  \end{equation*}
\end{theorem}

Though it produces weaker bounds than those of Ellenberg and Gijswijt in the cap set
problem, the argument of Kelley and Meka is still very interesting in the setting of
high-dimensional vector spaces over finite fields. Answering a question of Schoen and Sisask~\cite{SchoenSisask16}, Kelley and
Meka showed that if $A\subset\mathbf{F}_q^n$ has density $\alpha$, then its three-fold
sumset $A+A+A$ must contain an affine subspace of codimension $\ll\log(1/\alpha)^9$. Bloom
and Sisask~\cite{BloomSisask23} showed that their argument adapts to prove an integer
analogue of this result: if $A\subset[N]$ has density $\alpha$, then $A+A+A$ must contain
an arithmetic progression of length
$\gg N^{\Omega(1/\log(2/\alpha)^9)}\exp(-O(\log(2/\alpha)^2))$, which improves
substantially on the previous best lower bound of $\gg N^{\alpha^{1+o(1)}}$ due to
Sanders~\cite{Sanders08}. It is likely that the new ideas of Kelley and Meka will have
further applications in both the finite field model and integer settings. For more on the
work of Kelley and Meka, the interested reader should consult their
paper~\cite{KelleyMeka23} or the very short and clear exposition of their argument by
Bloom and Sisask~\cite{BloomSisask23}, which gives an almost self-contained proof (minus a
now standard almost-periodicity result and some basic facts about Bohr sets) of
Theorem~\ref{thm:KM} in less than ten pages.

The remainder of this section will be devoted to the cap set problem and the
Croot--Lev--Pach polynomial method.

\subsection{The slice rank method}

In this subsection, we will present a full proof of a slightly weaker version of Theorem~\ref{thm:EG}.
\begin{theorem}\label{thm:weakEG}
  If $A\subset\mathbf{F}_3^n$ is a cap set, then
  \[
    |A|\ll 2.838^n.
  \]
\end{theorem}
The improved power-saving bound in Theorem~\ref{thm:EG} can be obtained using essentially
the same argument, but with a bit more (tedious) work.

The proof of Theorem~\ref{thm:weakEG} uses the ``slice rank method'', which is a symmetric
reformulation, due to Tao~\cite{Tao16}, of the proof of Ellenberg and
Gijswijt~\cite{EllenbergGijswijt17}. The slice rank of a function is a measure of its
complexity, and the basic idea of the argument is the following. If $S$ is any finite set,
then the indicator function $1_{\Delta(S)}:S\times S\times S\to \mathbf{F}_3$ of the
diagonal $\Delta(S):=\{(s,s,s):s\in S\}$ of $S\times S\times S$ must have slice rank at
least $|S|$ (in the same way that the $N\times N$ identity matrix has rank $N$). On the
other hand, if $S\subset\mathbf{F}_3^n$ is a cap set, then $1_{\Delta(S)}$ can be seen to
have very small slice rank. Combining these two facts produces an upper bound for $|S|$.

Now we will define slice rank, specializing to functions taking values in $\mathbf{F}_3$.
\begin{Def}
  Let $k$ be a natural number and $S\neq\emptyset$ be a finite set.
  \begin{enumerate}
  \item A function $f:S^k\to\mathbf{F}_3$ has \defword{slice rank $1$} if there exist functions
    $g:S\to\mathbf{F}_3$ and $h:S^{k-1}\to\mathbf{F}$ and an index $i\in[k]$ such that
    \[
      f(x_1,\dots,x_k)=g(x_i)h(x_1,\dots,x_{i-1},x_{i+1},\dots,x_k)
    \]
    for all $x_1,\dots,x_k\in S$.
    \item A function $f:S^k\to\mathbf{F}_3$ has \defword{slice rank at most $m$} if there exist
      functions $f_1,\dots,f_m: S^k\to\mathbf{F}_3$, all with slice rank $1$, such that
      \[
        f=\sum_{j=1}^mf_j.
      \]
    \item The \defword{slice rank} of a function $f:S^k\to\mathbf{F}_3$ is defined to be the
    smallest natural number $m$ such that $f$ has slice rank at most $m$.
    \end{enumerate}
  \end{Def}
  Observe that, when $k=2$, the slice rank of $f$ is just the (usual) rank of the
  $|S|\times|S|$ matrix $(f(s,s'))_{s,s'\in S}$. This will be important in the proof of
  Lemma~\ref{lem:slicerankdiagonal} below.
  \begin{lemma}\label{lem:slicerankdiagonal}
    Let $S\neq\emptyset$ be a finite set. The slice rank of $1_{\Delta(S)}$ is $|S|$.
  \end{lemma}
  \begin{proof}
    First, observe that the slice rank of $1_{\Delta(S)}$ is at most $|S|$, since it can be
    written in the following way as the sum of slice rank $1$ functions from $S^3$ to $\mathbf{F}_3$:
    \[
      1_{\Delta(S)}(x,y,z)=\sum_{s\in S}\delta_s(x)\delta_s(y)\delta_s(z).
    \]
    So, the real content of this lemma is that $1_{\Delta(S)}$ cannot have slice rank
    strictly less than $|S|$, and the strategy is to reduce the problem of bounding the
    slice rank of $1_{\Delta(S)}$ from below to the problem of bounding the rank of a
    diagonal matrix from below, which is much easier to get our hands on. Suppose by way
    of contradiction that the slice rank of $1_{\Delta(S)}$ is less than $|S|$, so that there exists a natural number $n<|S|$,
    nonnegative integers $n_1$ and $n_2$, and functions $g_1,\dots,g_n:S\to\mathbf{F}_3$
    and $h_1,\dots,h_n:S^2\to\mathbf{F}_3$ for which
    \[
      1_{\Delta(S)}(x,y,z)=\sum_{j=1}^{n_1}g_j(x)h_j(y,z)+\sum_{j=n_1+1}^{n_2}g_j(y)h_j(x,z)+\sum_{j=n_2+1}^{n}g_j(z)h_j(x,y)
    \]
    where, we may assume without loss of generality that $0\leq n_1\leq n_2<n$. Set
    $n_3:=n-n_2$, so that $n_3$ is positive.

    Now, observe that if
    $r:S\to\mathbf{F}_3$ is any function, then the function $f_r:S^2\to\mathbf{F}_3$
    defined by
    \[
     f_r(x,y):=\sum_{z\in S}1_{\Delta(S)}(x,y,z)r(z)
   \]
   is supported on the diagonal of $S^2$, and takes the values $f_r(s,s)=r(s)$
   there. Thus, the slice rank of $f_r$ is exactly the size of the support of $r$. We will
   obtain a contradiction by finding a function $r:S\to\mathbf{F}$ with support of size
   greater than $n_2$ that is orthogonal to each of the functions $g_{n_2+1},\dots,g_{n}$,
   since then $f_r$ will have slice rank greater than $n_2$ but will be expressible in the form
   \[
     \sum_{j=1}^{n_1}g_j(x)\left(\sum_{z\in
         S}h_j(y,z)r(z)\right)+\sum_{j=n_1+1}^{n_2}g_j(y)\left(\sum_{z\in S}h_j(x,z)r(z)\right),
   \]
   which is the sum of at most $n_2$ functions of slice rank $1$, yielding a
   contradiction.

   Showing that such an $r:S\to\mathbf{F}_3$ exists just requires some simple linear
   algebra. Let
   \[
     V:=\left\{f:S\to\mathbf{F}_3:\sum_{s\in S}f(s)g_j(s)=0\text{ for all }j=n_2+1,\dots,n\right\}
   \]
   be the vector space over $\mathbf{F}_3$ of functions on $S$ that are orthogonal to each
   of $g_{n_2+1},\dots,g_n$, and suppose that $r\in V$ has maximal support size over all
   functions in $V$. Since $\codim V\leq n_3\leq n<|S|$, $V$ must contain some nonzero
   function, which means that $r$ is not identically zero. Further, the subspace $V'$ of
   $V$ of functions vanishing on the support $\supp{r}$ of $r$ has codimension at most $|\supp{r}|+n_3$. If
   $|\supp{r}|\leq n_2$, this means that $\codim{V'}\leq n_2+n_3=n<|S|$, and so there
   would exist a nonzero function $r'\in V'$, which necessarily has support disjoint from
   that of $r$. But then $r+r'$ would have support strictly larger than the support of
   $r$, which contradicts the maximality of $\supp{r}$. Thus, we must have
   $|\supp{r}|>n_2$, showing that a function with the properties desired above exists. We
   therefore conclude that $1_{\Delta(S)}$ has slice rank exactly $|S|$.
  \end{proof}

  Note that the above lemma required no information about $S$ aside from the assumption
  that it is finite and nonempty. In contrast, the next lemma relies crucially on the
  assumption that $A$ is a cap set.
\begin{lemma}
  If $A\subset\mathbf{F}_3^n$ is a cap set, then the slice rank of $1_{\Delta(A)}$ is
  bounded above by
  \begin{equation}\label{eq:Mbound}
    3\cdot \#\left\{\mathbf{a}\in\{0,1,2\}^n:\sum_{i=1}^na_i\leq \frac{2n}{3}\right\}.
  \end{equation}
\end{lemma}
\begin{proof}
  Note that
  \[
    1_{\Delta(A)}(x,y,z)=\delta_0(x+y+z)
  \]
  since $A$ is a cap set, and a cap set contains only trivial solutions to $x+y+z=0$
  (which, in characteristic $3$, is the same as the equation $x+z=2y$ characterizing
  three-term arithmetic progressions). Thus, we can express $1_{\Delta(A)}$ as
  \[
    1_{\Delta(A)}(x,y,z)=\prod_{i=1}^n\left(1-(x_i+y_i+z_i)^2\right)
  \]
  since $\delta_0(w)=1-w^2$ for all $w\in\mathbf{F}_3$. The polynomial on the right-hand
  side above has degree $2n$, and every monomial appearing in it takes the form
  \begin{equation*}
    \prod_{i=1}^nx_i^{a_i}y_i^{b_i}z_i^{c_i},
  \end{equation*}
  where $a_i,b_i,c_i\in\{0,1,2\}$ for each $i\in [n]$ and
  $\sum_{i=1}^n(a_i+b_i+c_i)\leq 2n$. It follows from this second fact that, for any such
  monomial, $\sum_{i=1}^na_i$, $\sum_{i=1}^n b_i$, or $\sum_{i=1}^nc_i$ is at most
  $2n/3$ (for otherwise $\sum_{i=1}^n(a_i+b_i+c_i)$ would be greater than $2n$). This means that there exist functions
  $g^{(1)}_{\mathbf{a}},g^{(2)}_{\mathbf{b}},g^{(3)}_{\mathbf{c}}:A^2\to \mathbf{F}_3$
  such that
\begin{align*}
  1_{\Delta(A)}(x,y,z)=&\sum_{\substack{\mathbf{a}\in\{0,1,2\}^n \\ \sum_{i=1}^n a_i\leq
  2n/3}}x_1^{a_1}\cdots x_n^{a_n}g^{(1)}_{\mathbf{a}}(y,z)+\sum_{\substack{\mathbf{b}\in\{0,1,2\}^n \\ \sum_{i=1}^n b_i\leq
  2n/3}}y_1^{b_1}\cdots y_n^{b_n}g^{(2)}_{\mathbf{b}}(x,z)\\
  &+\sum_{\substack{\mathbf{c}\in\{0,1,2\}^n \\ \sum_{i=1}^n c_i\leq
  2n/3}}z_1^{c_1}\cdots z_n^{c_n}g^{(3)}_{\mathbf{c}}(x,y),
\end{align*}
from which is follows that $1_{\Delta(A)}$ has slice rank at most
\[
      3\cdot \#\left\{\mathbf{a}\in\{0,1,2\}^n:\sum_{i=1}^na_i\leq \frac{2n}{3}\right\}.
\]
\end{proof}

Now we can prove Theorem~\ref{thm:weakEG}.
\begin{proof}
  The above two lemmas say that the slice rank of $1_{\Delta(A)}$ is both equal to
  $|A|$ and bounded above by~\eqref{eq:Mbound}. Thus,
  \[
    |A|\ll \#\left\{\mathbf{a}\in\{0,1,2\}^n:\sum_{i=1}^na_i\leq \frac{2n}{3}\right\},
  \]
  and it remains to bound the right-hand side of the above. Note that this is exactly
  $3^{-n}$ times the
  probability that, if $X_1,\dots,X_n$ is a sequence of independent uniform random
  variables taking values in $\{0,1,2\}$, the random variable $X_1+\dots+X_n$ is
  at most $2n/3$. Hoeffding's inequality says that this probability is $\ll e^{-n/18}$,
  so that
  \[
    |A|\ll \left(\frac{3}{e^{1/18}}\right)^n\ll 2.838^n,
  \]
  as desired.
\end{proof}
To improve the upper bound to $\ll 2.756^n$ as in Theorem~\ref{thm:EG}, one just needs to
put a bit more effort into estimating the size of~\eqref{eq:Mbound}.

\subsection{Further results and questions}\label{ssec:further}

The slice rank method has had numerous applications to a wide variety of problems in
combinatorics since the work of Croot--Lev--Pach and Ellenberg--Gijswijt. We will
mention a few that are relevant to the theme of this survey.

A \defword{tricolored sum-free set} in $\mathbf{F}_p^n$ is a collection of triples
$ \left((x_i,y_i,z_i)\right)_{i=1}^{M}$ of elements of $\mathbf{F}_p^n$ such that
$x_i+y_j+z_k=0$ if and only if $i=j=k$. Note that if $A\subset\mathbf{F}_3^n$ is a
cap set, then the diagonal $\Delta(A)$ is a tricolored sum-free set in
$\mathbf{F}_3^n$. Blasiak, Church, Cohn, Grochow, Naslund, Sawin, and
Umans~\cite{BlasiakChurchCohnGrochowNaslundSawinUmans17} and, independently, Alon,
observed that the slice rank method can also be used to bound the size of tricolored
sum-free sets in $\mathbf{F}_p^n$, thus generalizing Theorem~\ref{thm:EG}. Kleinberg,
Sawin, and Speyer~\cite{KleinbergSawinSpeyer18} showed that the bound obtained
in~\cite{BlasiakChurchCohnGrochowNaslundSawinUmans17} is essentially optimal for each
prime $p$, assuming a conjecture that was shortly after proved by Norin~\cite{Norin19} and
Pebody~\cite{Pebody18}. These results combined yield the following theorem.
\begin{theorem}\label{thm:tricolored}
  For every prime $p$, there exists a constant $c_p\in(0,1)$ such that the following
  holds. Any tricolored sum-free set in $\mathbf{F}_p^n$ has size at most
  $p^{(1-c_p)n}$, and there exists a tricolored sum-free set in $\mathbf{F}_p^n$
  of size at least $p^{(1-c_p)n-o(1)}$.
\end{theorem}
The largest known cap sets in $\mathbf{F}_3^n$ have size on the order of
$2.218^n$~\cite{Tyrrell22}, so, in contrast to the tricolored sum-free set problem, there
is still an exponential gap between the best known bounds in the cap set problem.

Recall that the triangle removal lemma, a standard result in extremal graph theory, states
that for every $\varepsilon>0$ there exists a $\delta>0$ such that any graph on $n$
vertices containing $\delta n^3$ triangles (i.e., copies of the complete graph on three
vertices) can be made triangle-free by removing at most $\varepsilon n^2$ edges. The
quickest way to prove the triangle removal lemma is by using Szemer\'edi's regularity
lemma, which produces an upper bound for $\frac{1}{\delta}$ that is a tower of height
polynomial in $\frac{1}{\varepsilon}$. The best known bounds for $\frac{1}{\delta}$ in the
triangle removal lemma are due to Fox~\cite{Fox11}, who showed that one can take
$\frac{1}{\delta}$ bounded by a tower of height $\ll\log(1/\varepsilon)$.

Now let $A$, $B$, and $C$ be subsets of a finite abelian group $(G,+)$. A
\textit{triangle} in $A\times B\times C$ is a triple $(x,y,z)\in A\times B\times C$
satisfying $x+y+z=0$. Green~\cite{Green05regularity} developed an arithmetic regularity
lemma, and used it to prove the following ``arithmetic triangle removal lemma'': for every
$\varepsilon>0$, there exists a $\delta>0$ such that if $A\times B\times C$ contains
$\delta |G|^2$ triangles, then $A\times B\times C$ can be made triangle-free by removing
at most $\varepsilon |G|$ elements from $A$, $B$, and $C$. Green's argument produced a
bound for $\frac{1}{\delta}$ that is a tower of height polynomial in
$\frac{1}{\varepsilon}$. Kr\'al, Serra, and Vena~\cite{KralSerraVena09} observed that
Green's arithmetic triangle removal lemma also follows from the triangle removal lemma for
graphs, and used this observation to prove a generalization of Green's result for
nonabelian groups. Indeed, consider the graph with vertex set consisting of three copies
of $G$,
\begin{equation*}
  V:= \left(G\times\{1\}\right)\cup \left(G\times\{2\}\right)\cup \left(G\times\{3\}\right),
\end{equation*}
and edge set
\begin{align*}
  E:=&\left\{((g,1),(g+a,2)):g\in G, a\in A\right\}\cup\left\{((g,2),(g+b,3)):g\in G, b\in
       B\right\}\\
  &\cup \left\{((g,3),(g+c,1)):g\in G,c\in C\right\}.
\end{align*}
Then a triangle in the graph $(V,E)$ is exactly a triple of the form
\[
  ((g,1),(g+a,2)),((g+a,2),(g+a+b,3)),((g+a+b,3),(g+a+b+c,1)),
\]
where $g\in G$, $(a,b,c)\in A\times B\times C$, and $a+b+c=0$. Thus, the number of
triangles in the graph $(V,E)$ equals $|G|$ (the number of possible choices for $g\in G$)
times the number of triangles in $A\times B\times C$. The arithmetic triangle removal
lemma now follows from the triangle removal lemma for graphs and the pigeonhole
principle. Thus, Fox's improved bound in the triangle removal lemma leads to the same
improved bound in the arithmetic triangle removal lemma.

Green asked in~\cite{Green05regularity} whether a polynomial bound in the arithmetic
triangle removal lemma could hold in the setting of high-dimensional vector spaces over
finite fields. This was answered in the affirmative by Fox and
L. M. Lov\'asz~\cite{FoxLovasz17}, who obtained a massive improvement over the previous
best known bound, and also show that their result is essentially tight.
\begin{theorem}\label{thm:arithmetictriangleremoval}
  For every prime $p$, there exists a constant $C_p\in(0,1)$ such that the following
  holds. If $\varepsilon>0$ and $\delta=(\varepsilon/3)^{C_p}$, then whenever
  $A,B,C\subset \mathbf{F}_p^n$ are such that $A\times B\times C$ has at most
  $\delta p^{2n}$ triangles, then $A\times B\times C$ can be made triangle-free by
  removing at most $\varepsilon p^n$ elements from $A$, $B$, and $C$. The best
  possible $\delta$ for which this result holds satisfies $\delta\leq \varepsilon^{C_p-o(1)}$.
\end{theorem}
Fox and Lov\'asz prove Theorem~\ref{thm:arithmetictriangleremoval} by a clever argument
using Theorem~\ref{thm:tricolored} as its key input. In~\cite{Green05regularity}, Green
actually proved a more general arithmetic $k$-cycle removal lemma. A polynomial bound in
this theorem was later shown by Fox, L. M. Lov\'asz, and Sauermann~\cite{FoxLovaszSauermann18}.

We will now turn to the problem of finding nonlinear configurations, both in subsets of
the integers and in various finite field model settings. Let $P\in\mathbf{Z}[y]$ be a
polynomial with zero constant term and degree greater than $1$. Proving a conjecture of
L. Lova\'sz (which was also confirmed independently by Furstenberg~\cite{Furstenberg77},
though without quantitative bounds), S\'ark\"ozy~\cite{Sarkozy78II} showed in 1978 that if
$A\subset[N]$ contains no nontrivial two-term polynomial progressions $x,x+P(y)$, then
$|A|\ll\frac{N}{\log\log{N}}$. Analogously to Question~\ref{qn:kAP}, the following problem
is of great interest, in particular for the simplest nonlinear case $P(y)=y^2$.
\begin{qn}\label{qn:Sarkozy}
  How large can a subset of the first $N$ integers be if it contains no nontrivial
  progressions of the form $x,x+P(y)$?
\end{qn}
There is a huge gap between the best-known upper and lower bounds for
Question~\ref{qn:Sarkozy}. Bloom and Maynard~\cite{BloomMaynard22} showed that any
$A\subset [N]$ lacking the progression $x,x+y^2$ must satisfy
$|A|\ll\frac{N}{(\log{N})^{C\log\log\log{N}}}$ for some constant $C>0$, improving on a
bound of Balog, Pintz, Steiger, and Szemer\'edi~\cite{BalogPelikanPintzSzemeredi94} that
was almost thirty years old. Arala~\cite{Arala2023} extended the argument of Bloom and
Maynard to prove bounds of the same shape for sets lacking the progression $x,x+P(y)$ for
any $P$ having zero constant term (and, even more generally, to any intersective $P$). The
largest known subsets of $[N]$ lacking two-term polynomial progressions (which come from
either a greedy construction or a construction of Ruzsa~\cite{Ruzsa84}) all have size on
the order of $N^{\gamma_P}$ for some $\gamma_P\in(0,1)$. I think it is more likely that
these lower bounds are closer to the truth in Question~\ref{qn:Sarkozy}, especially given
Green's recent proof of power-saving bounds for subsets of $[N]$ whose difference sets
avoid the set of shifted primes $\{p-1:p\text{ prime}\}$~\cite{Green22}.

There are a few natural settings involving finite fields in which one can formulate a
version of Question~\ref{qn:Sarkozy}. Using the Weil bound for additive character sums, it
is not difficult to prove power-saving bounds for the size of subsets of $\mathbf{F}_p$
lacking any fixed polynomial progression $x,x+P(y)$. So, replacing $[N]$ by $\mathbf{F}_p$
in Question~\ref{qn:Sarkozy} leads to too simple a problem (though, as we will discuss
later in this survey, the analogous question for longer polynomial progressions in subsets
of finite fields is much more interesting).

Question~\ref{qn:Sarkozy} makes sense and is highly nontrivial in the \defword{function
  field setting}, where $\mathbf{Z}$ is replaced by $\mathbf{F}_q[t]$ for a fixed prime
power $q$ and $\{1,\dots,N\}$ is replaced by
$\mathbf{F}_q[t]_{<n}:=\{f\in\mathbf{F}_q[t]:\deg{f}<n\}$. Observe that
$\left|\mathbf{F}_q[t]_{<n}\right|=q^n$.
\begin{qn}\label{qn:ffSarkozy}
  Let $P\in\mathbf{F}_q[X]$ have zero constant term. How large can a subset of $\mathbf{F}_q[t]_{<n}$ be if it contains no nontrivial
  progressions of the form $f,f+P(g)$?
\end{qn}
This question was first studied by Le and Liu~\cite{LeLiu2013}, who proved a bound of the form
$|A|\ll\frac{q^n(\log{n})^{O(1)}}{n}$ when $A\subset \mathbf{F}_q[t]_{<n}$ lacks the progression $f,f+g^2$.

Shortly after the proof of Theorem~\ref{thm:EG}, Green~\cite{Green17} used the slice rank
method to prove power-saving bounds in Question~\ref{qn:ffSarkozy}, provided that the
polynomial $P$ satisfies certain technical conditions. These conditions were removed in
recent work of Li and Sauermann~\cite{LiSauermann22}, who also used the slice rank method.
 \begin{theorem}\label{thm:LS}
   Fix a prime power $q$, and let $P\in\mathbf{F}_q[X]$ have zero constant term. If
   $A\subset \mathbf{F}_q[t]_{<n}$ has no nontrivial progressions
   \begin{equation*}
     f,f+P(g),
   \end{equation*}
   then
   \[
     |A|\ll_{q,\deg{P}}(q^n)^{1-\gamma_{q,\deg{P}}},
   \]
   for some constant $1>\gamma_{q,\deg{P}}>0$.
 \end{theorem}
 This theorem implies power-saving bounds for S\'ark\"ozy's theorem in $\mathbf{F}_{p^n}$
 in the regime where the prime $p$ is fixed and $n$ tends to infinity.
 \begin{cor}
 Fix a prime $p$ and let $P\in\mathbf{F}_{p^n}[x]$ have zero constant term. If $A\subset\mathbf{F}_{p^n}$ has
 no nontrivial progressions
 \[
   x,x+P(y),
 \]
 then
 \[
   |A|\ll_{p,\deg{P}}(p^n)^{1-\gamma_{p,\deg{P}}}.
 \]
\end{cor}

Perhaps the most important open question regarding the slice rank method is whether it, or
some variant, can be used to prove power-saving bounds for subsets of $\mathbf{F}_5^n$ lacking four-term
arithmetic progressions. The best-known bound for
subsets of $\mathbf{F}_5^n$ free of four-term arithmetic progressions is due to Green and
Tao~\cite{GreenTao09} (see also the corrected arXiv version~\cite{GreenTao2012}).
 \begin{theorem}\label{thm:4AP}
   If $A\subset\mathbf{F}_5^n$ contains no nontrivial four-term arithmetic progressions,
   then
   \[
     |A|\ll\frac{p^n}{n^c},
   \]
   where $c>0$ is a (very small) absolute constant.
 \end{theorem}
 Any substantial improvement to the bound in Theorem~\ref{thm:4AP} would be of great
 interest.

 It is also not known whether the slice rank method, or some variant of it, could apply to
 prove power-saving bounds for subsets of $(\mathbf{F}_2^n)^2$ lacking nontrivial
 \defword{corners},
 \begin{equation*}
     (x,y),(x,y+z),(x+z,y) \qquad (z\neq 0).
   \end{equation*}
   The best-known bounds for corner-free sets comes
   from adapting a proof of Shkredov~\cite{Shkredov06I,Shkredov06II} in the integer
   setting to the finite field model setting~\cite{Green05}, with the largest exponent of
   $\log{n}$ coming from work of Lacey and McClain~\cite{LaceyMcClain07}.
 \begin{theorem}\label{thm:corners}
   If $A\subset(\mathbf{F}_2^n)^2$ contains no nontrivial corners, then
   \[
     |A|\ll\frac{4^{n}\log\log{n}}{\log{n}}.
   \]
 \end{theorem}
 The problem of adapting the slice rank method to prove power-saving bounds for sets
 lacking four-term arithmetic progressions or for sets lacking corners has received a
 considerable amount of attention, and is likely very difficult. However, I am optimistic
 that some of the ideas from the recent breakthrough of Kelley and
 Meka~\cite{KelleyMeka23} on Roth's theorem in the integer setting could be used to
 improve the bounds for corner-free sets, in both the integer and finite field model
 setting. In particular, I expect the following problem should be attackable.
 \begin{problem}
   Show that if $A\subset(\mathbf{F}_2^n)^2$ contains no nontrivial corners, then
   \[
     |A|\ll\frac{4^n}{n^{c'}}
   \]
   for some absolute constant $c'>0$.
 \end{problem}

\section{The inverse theorems for the Gowers uniformity norms}
Confirming a conjecture of Erd\H{o}s and Tur\'an from 1936, Szemer\'edi~\cite{Szemeredi75}
proved in 1975 that if $A\subset[N]$ contains no nontrivial $k$-term arithmetic
progressions, then $|A|=o_k(N)$. Answering Question~\ref{qn:kAP} is therefore equivalent
to determining the best possible bounds for $|A|$ in Szemer\'edi's theorem. The bounds for
$o_k(N)$ that can be extracted from Szemer\'edi's argument, in which he introduced his
now-famous regularity lemma for graphs, are extremely weak--the savings over the trivial bound
of $N$ are of inverse-Ackermann type.

No reasonable bounds in Question~\ref{qn:kAP}, i.e., a savings over the trivial bound of
$N$ that grows at least as fast as a finite number of iterated logarithms of $N$, were
known for any $k$ larger than three until pioneering work of
Gowers~\cite{Gowers98,Gowers01} in the late 1990s and early 2000s, who initiated the study
of ``higher-order Fourier analysis'' and used it to prove that
\[
|A|\ll_k\frac{N}{(\log\log{N})^{2^{-2^{k+9}}}}.
\]
for any $k\geq 4$. Before we explain what higher-order Fourier analysis is, we will briefly illustrate why
Fourier analysis is relevant to the study of three-term arithmetic progressions.

\subsection{The Fourier-analytic approach to Roth's theorem}

Let $A\subset\mathbf{F}_3^n$ be a nonempty subset of density $\alpha$. Note that if we
construct a random subset $A'$ of $\mathbf{F}_3^n$ by including each element independently
and uniformly with probability $\alpha$, then $A'$ will almost always have density very
close to $\alpha$ and contain very close to $\alpha^2|\mathbf{F}_3^n|^2$ three-term
arithmetic progressions. The distance between the number of three-term arithmetic
progressions in $A$ and the number $\alpha^2|\mathbf{F}_3^n|^2$ expected in a random set
of the same density can be controlled using Fourier analysis.

For any functions $f_0,f_1,f_2:\mathbf{F}_3^n\to\mathbf{C}$, we define the trilinear
average $\Lambda_3(f_0,f_1,f_2)$ by
\[
  \Lambda_3(f_0,f_1,f_2):=\mathbf{E}_{x,y\in\mathbf{F}_3^n}f_0(x)f_1(x+y)f_2(x+2y).
\]
Thus, $\Lambda_3(1_A,1_A,1_A)$ equals the normalized count of the number of three-term
arithmetic progressions in $A$. Observe that, by Fourier inversion, $\Lambda_3(1_A,1_A,1_A)$ also equals
\begin{equation*}
 \sum_{\xi_1,\xi_2,\xi_3\in\mathbf{F}_3^n}\widehat{1_A}(\xi_1)\widehat{1_A}(\xi_2)\widehat{1_A}(\xi_3)\left(\mathbf{E}_{x,y\in\mathbf{F}_3^n}e_3\left([\xi_1+\xi_2+\xi_3]x+[\xi_2+2\xi_3]y\right)\right),
\end{equation*}
which, by orthogonality of characters and Parseval's identity, equals
\begin{equation*}
                        \sum_{\xi\in\mathbf{F}_3^n}\widehat{1_A}(\xi)^2\widehat{1_A}(-2\xi)= \alpha^3+O\left(\max_{0\neq\xi\in\mathbf{F}_3^n}\left|\widehat{1_A}(\xi)\right|\right).
\end{equation*}
Thus, if $A$ has far from the $\alpha^3|\mathbf{F}_3^n|^2$ three-term arithmetic
progressions expected in a random set of density $\alpha$ (which happens, for example,
when $A$ is a cap set and $n$ is large enough that the number of trivial three-term
progressions is significantly smaller than $\alpha^3|\mathbf{F}_3^n|^2$), then
$\widehat{1_A}(\xi)$ must be large for some nonzero $\xi\in\mathbf{F}_3^n$. This gives us
strong structural information about $A$, which can be used to continue the argument--full
details can be found in the surveys of Green~\cite{Green05} and Wolf~\cite{Wolf15}, or in
my Bourbaki seminar article on quantitative bounds in Roth's theorem~\cite{Peluse22}.

Fourier analysis is not sufficient for the study of four-term, and longer, arithmetic
progressions, in the sense that a set can have no large nontrivial Fourier coefficients
but still contain far from the number of four-term arithmetic progressions expected in a
random set of the same size. Indeed, consider the set $S\subset\mathbf{F}_5^n$ defined by
\[
  S:=\{x\in\mathbf{F}_5^n:x\cdot x=0\}.
\]
It is easy to verify that $S$ has density $\frac{1}{5}+O(\frac{1}{\sqrt{5^n}})$ and
all nontrivial Fourier coefficients satisfy
$\left| \widehat{1}_S(\xi)\right|\ll \frac{1}{\sqrt{5^n}}$. However, since
\[
  x\cdot x-3(x+y)\cdot(x+y)+3(x+2y)\cdot(x+2y)-(x+3y)\cdot(x+3y)=0
\]
for all $x,y\in \mathbf{F}_5^n$, if the first three terms of a four-term arithmetic progression lie in
$S$, then the last term is forced to as well. This means that $S$ contains
$\left(\frac{1}{125}+O(\frac{1}{\sqrt{5^n}})\right) |\mathbf{F}_5^n|^2$ four-term
arithmetic progressions, many more than the roughly $\frac{1}{625}|\mathbf{F}_5^n|^2$ expected
in a random set of density $\frac{1}{5}$. Thus, something beyond Fourier analysis is
needed to get a handle on four-term arithmetic progressions.

\subsection{Higher-order Fourier analysis}

Let $(G,+)$ be a finite abelian group (which, for us, will
be either $\mathbf{Z}/N\mathbf{Z}$ or $\mathbf{F}_p^n$) and $s$ be a natural number. For any
complex-valued function $f$ on $G$ and $h\in G$, the \defword{multiplicative
  discrete derivative} $\Delta_hf:G\to\mathbf{C}$ is defined by
\begin{equation*}
  \Delta_{h}f(x):=f(x)\overline{f(x+h)}.
\end{equation*}
For $h_1,\dots,h_s\in G$, we denote the $s$-fold multiplicative discrete derivative of $f$
by
\begin{equation*}
  \Delta_{h_1,\dots,h_s}f:=\Delta_{h_1}\cdots\Delta_{h_s}f.
\end{equation*}
Observe that $\Delta_{h_1,\dots,h_s}f=\Delta_{h_{\sigma(1)},\dots,h_{\sigma(s)}}f$ for any
permutation $\sigma$ of $\{1,\dots,s\}$, so $\Delta_{h_1,\dots,h_s}$ depends only on the
multiset of differences $\{h_1,\dots,h_s\}$. We can similarly define the \defword{additive
  discrete derivative} of a function $g:G\to G'$, where $(G',+)$ is also an abelian group,
by
\[
  \partial_hg(x)=g(x)-g(x+h),
\]
as well as the $s$-fold additive discrete derivative
\begin{equation*}
  \partial_{h_1,\dots,h_s}g=\partial_{h_1}\cdots\partial_{h_s}g.
\end{equation*}

For any natural number $s$, the \defword{Gowers $U^s$-norm} $\|f\|_{U^s}$ of $f$ is defined by
\begin{equation*}
  \|f\|_{U^s}:=\left(\mathbf{E}_{x,h_1,\dots,h_s\in G}\Delta_{h_1,\dots,h_s}f(x)\right)^{1/2^s}.
\end{equation*}
(note that $\mathbf{E}_{x,h_1,\dots,h_s\in G}\Delta_{h_1,\dots,h_s}f(x)$ is always real-valued and nonnegative, and
thus possesses a nonnegative $2^s$-th root). For example, $\|f\|_{U^1}=\left|\mathbf{E}_{x\in G}f(x)\right|$ and
\begin{equation*}
\|f\|_{U^2}^4=  \mathbf{E}_{x,h_1,h_2\in G}f(x)\overline{f(x+h_1)}\overline{f(x+h_2)}f(x+h_1+h_2).
\end{equation*}
We will now list some basic facts about these norms.
\begin{lemma}\label{lem:basic}
  Let $f:G\to\mathbf{C}$.
  \begin{enumerate}
  \item $\|\cdot\|_{U^1}$ is a seminorm.
  \item $\|\cdot\|_{U^s}$ is a norm when $s\geq 2$.
  \item $\|f\|_{U^1}\leq\|f\|_{U^2}\leq\dots\leq \|f\|_{U^s}\leq \|f\|_{U^{s+1}}\leq\dots\leq\|f\|_{\ell^\infty}$.
  \item $\|f\|_{U^{s+1}}^{2^{s+1}}=\mathbf{E}_{h\in G}\|\Delta_{h}f\|_{U^{s}}^{2^s}$ for
    all $s\geq 1$.
  \item $\|f\|_{U^2}=\|\widehat{f}\|_{\ell^4}$.
  \end{enumerate}
\end{lemma}
None of these statements are very hard to prove, but hints can be found in~\cite{Green07}
and~\cite{Tao12}, for example. While the $U^2$-norm is simply the $\ell^4$-norm of the
Fourier transform of $f$, the study of the $U^s$-norms when $s\geq 3$ is called
\defword{higher-order Fourier analysis}.

Gowers observed (in the integer setting) that the $U^s$-norm controls the count of
$(s+1)$-term arithmetic progressions in subsets of abelian groups. We will present a proof
of this statement in the setting of vector spaces over finite fields.
\begin{lemma}\label{lem:gvnAP}
  Let $p\geq k\geq 2$ and $f_0,\dots,f_{k-1}:\mathbf{F}_p^n\to\mathbf{C}$ be $1$-bounded
  functions. Then
  \begin{equation*}
    \left|\mathbf{E}_{x,y\in G}f_0(x)f_1(x+y)\cdots f_{k-1}(x+(k-1)y)\right|\leq\|f_{k-1}\|_{U^{k-1}}.
  \end{equation*}
\end{lemma}
With a bit more work, one can replace the right-hand side with the stronger bound
$\min_{0\leq i\leq k-1}\|f_i\|_{U^{k-1}}$.
\begin{proof}
  We proceed by induction on $k$, starting with the base case $k=2$. Observe, by making
  the change of variables $y\mapsto y-x$, that
  \begin{equation*}
    \left|\mathbf{E}_{x,y\in \mathbf{F}_p^n}f_0(x)f_1(x+y)\right|=\left|\mathbf{E}_{x\in
        \mathbf{F}_p^n}f_0(x)\right|\left|\mathbf{E}_{y\in \mathbf{F}_p^n}f_1(y)\right|\leq \|f_1\|_{U^1}.
  \end{equation*}
  Now suppose that we have proven the result for a general $k\geq 2$. Writing
  \begin{equation}\label{eq:swapsum}
    \left|\mathbf{E}_{x,y\in \mathbf{F}_p^n}f_0(x)f_1(x+y)\cdots f_{k}(x+ky)\right|^2
  \end{equation}
  as
  \begin{equation*}
    \left|\mathbf{E}_{x\in
        \mathbf{F}_p^n}f_0(x)\left(\mathbf{E}_{y\in
          \mathbf{F}_p^n}f_1(x+y)\cdots
        f_{k}(x+ky)\right)\right|^2,
  \end{equation*}
  we have by that Cauchy--Schwarz inequality that~\eqref{eq:swapsum} is bounded above by
  \begin{equation*}
    \mathbf{E}_{x,y,h\in \mathbf{F}_p^n}\Delta_{h}f_1(x+y)\cdots \Delta_{kh}f_k(x+ky),
  \end{equation*}
  which equals
  \begin{equation*}
    \mathbf{E}_{x,y,h\in \mathbf{F}_p^n}\Delta_{h}f_1(x)\cdots \Delta_{kh}f_k(x+(k-1)y)
  \end{equation*}
  by making the change of variables $x\mapsto x-y$. It now follows from the induction
  hypothesis and our assumption that $p\geq k+1$ that
\begin{equation*}
  \left|\mathbf{E}_{x,y\in \mathbf{F}_p^n}f_0(x)f_1(x+y)\cdots f_{k}(x+ky)\right|^2\leq
  \mathbf{E}_{h\in \mathbf{F}_p^n}\|\Delta_{kh}f_{k}\|_{U^{k-1}}=\mathbf{E}_{h\in \mathbf{F}_p^n}\|\Delta_{h}f_{k}\|_{U^{k-1}}.
\end{equation*}
By H\"older's inequality,
\begin{equation*}
  \left|\mathbf{E}_{x,y\in \mathbf{F}_p^n}f_0(x)f_1(x+y)\cdots f_{k}(x+ky)\right|^{2^{k}}\leq \mathbf{E}_{h\in \mathbf{F}_p^n}\|\Delta_{h}f_{k}\|_{U^{k-1}}^{2^{k-1}}=\|f_k\|_{U^k},
\end{equation*}
completing the inductive step.
\end{proof}

An important consequence is that the $U^s$-norms are measures of pseudorandomness as far
as counting $k$-term arithmetic progressions is concerned. Let
$A\subset \mathbf{F}_p^n$ have density $\alpha$ in $\mathbf{F}_p^n$ and set
$f_A:=1_A-\alpha$. Then, by Lemma~\ref{lem:gvnAP}, we have
\begin{equation*}
  \left|\frac{\#\{(x,y)\in (\mathbf{F}_p^n)^2:x,x+y,\dots,x+(k-1)y\in A\}}{|\mathbf{F}_p^n|^2}-\alpha^k\right|\ll_k\|f_A\|_{U^{k-1}}.
\end{equation*}
Thus, if $A$ has far from the density $\alpha^k$ of $k$-term arithmetic progressions
expected in a random subset of density $\alpha$, then $\|f_A\|_{U^{k-1}}$ is large. To
make use of this information, we need an \defword{inverse theorem} for the $U^{k-1}$-norm,
i.e., a structural result for bounded functions with large $U^{k-1}$-norm. When $k=3$ in
this analysis, so that $\|f_A\|_{U^2}$ is large, then it is easy to deduce that $1_A$ must
have some large nontrivial Fourier coefficient.
\begin{lemma}[$U^2$-inverse theorem]
Let $f:\mathbf{F}_p^n\to\mathbf{C}$ be a $1$-bounded function. Then
\begin{equation*}
  \left\|f\right\|_{U^2}^2\leq\|\widehat{f}\|_{\ell^\infty}.
\end{equation*}
\end{lemma}
\begin{proof}
  By the last statement of Lemma~\ref{lem:basic},
  \begin{equation*}
    \|f\|_{U^2}^4=\sum_{\xi\in\mathbf{F}_p^n}\left|\widehat{f}(\xi)\right|^4.
  \end{equation*}
  So,
  \begin{equation*}
    \|f\|_{U^2}^4\leq\max_{\xi\in\mathbf{F}_p^n}\left|\widehat{f}(\xi)\right|^2\sum_{\xi\in\mathbf{F}_p^n}\left|\widehat{f}(\xi)\right|^2\leq
    \max_{\xi\in\mathbf{F}_p^n}\left|\widehat{f}(\xi)\right|^2\cdot\mathbf{E}_{x\in\mathbf{F}_p^n}\left|f(x)\right|^2\leq\|\widehat{f}\|_{\ell^\infty}^2,
  \end{equation*}
  since $f$ is $1$-bounded. Taking the square root of both sides gives the conclusion of
  the lemma.
\end{proof}

To prove Szemer\'edi's theorem for sets lacking progressions of length greater than three,
Gowers proved a ``local'' inverse theorem for the $U^{s}$-norm on cyclic groups. This
says, roughly speaking, that if $f$ is $1$-bounded and $\|f\|_{U^s}$ is large, then there
exists a partition of $\mathbf{Z}/N\mathbf{Z}$ into long arithmetic progressions
$I_1,\dots,I_k$ such that, on average, $f$ has large correlation on $I_j$ with a
polynomial phase $e(P(x))$ of degree $\deg{P}\leq s-1$. A ``global'' inverse theorem on
cyclic groups was not proved until several years later, first in the case $s=3$ by Green
and Tao~\cite{GreenTao08} and then in general by Green, Tao, and
Ziegler~\cite{GreenTaoZiegler11,GreenTaoZiegler12}, who showed that $f$ must have large
correlation over the entire group with an $(s-1)$-step nilsequence of bounded
complexity. This work of Green, Tao, and Ziegler was purely qualitative, and, more
recently, Manners~\cite{Manners18} proved the first fully general quantitative version of
the global inverse theorems for the Gowers norms.  Defining nilsequences and discussing
the integer setting further would take us too far from the scope of this article, so for
the remainder of this section we will discuss the $U^s$-inverse theorems in the setting of
high-dimensional vector spaces over finite fields.

\subsection{The $U^3$-inverse theorem and the polynomial Freiman--Ruzsa conjecture}
Based on the early arguments of Gowers in cyclic groups, Samorodnitsky~\cite{Samorodnitsky07} (for $p=2$) and Green and Tao~\cite{GreenTao08}
(for $p>2$) proved the following $U^3$-inverse theorem in the finite field model setting.
\begin{theorem}\label{thm:U3inverse}
  Fix a prime $p$. There exists a constant $c>0$ such that if $f:\mathbf{F}_p^n\to \mathbf{C}$ is $1$-bounded and $\|f\|_{U^3}\geq \delta$, then
  there exists a polynomial $Q:\mathbf{F}_p^n\to\mathbf{F}_p$ of degree at most $2$ for which
\[
  \left|\mathbf{E}_{x\in\mathbf{F}_p^n}f(x)e_p(Q(x))\right|\geq\exp(-c\delta^{-c}).
\]
\end{theorem}
Sanders~\cite{Sanders12b} improved the lower bound for the correlation of $f$ with a
quadratic phase in Theorem~\ref{thm:U3inverse} to
\[
  \left|\mathbf{E}_{x\in\mathbf{F}_p^n}f(x)e_p(Q(x))\right|\geq\exp\left(-c(\log(2/\delta))^c\right)
\]
for some constant $c>0$ by proving quasipolynomial bounds in Bogolyubov's theorem, a key
ingredient in the proof of Theorem~\ref{thm:U3inverse}.  Green and Tao~\cite{GreenTao10}
and, independently, Lovett~\cite{Lovett2012} showed that Theorem~\ref{thm:U3inverse}
holding with a lower bound depending only polynomially on $\delta$ is equivalent to the
polynomial Freiman--Ruzsa conjecture, one of the most important conjectures in additive
combinatorics.
\begin{conj}\label{conj:PFR}
  There exists an absolute constant $c>0$ such that the following holds. If
  $A\subset\mathbf{F}_p^n$ has small doubling $|A+A|\leq K|A|$, then there exists an affine
  subspace $V\leq\mathbf{F}_p^n$ size $|V|\ll K^{c}|A|$ such that
  $|A\cap V|\gg K^{-c}|A|$.
\end{conj}
Several equivalent formulations of Conjecture~\ref{conj:PFR} can be found at the end of Green's
survey~\cite{Green05}, with proofs of the equivalences located in Green's accompanying
notes~\cite{Green05note2}. In addition to implying polynomial bounds in
Theorem~\ref{thm:U3inverse}, the $p=2$ case of the polynomial Freiman--Ruzsa conjecture
has numerous applications in theoretical computer science, a list of which can be found in
Lovett's exposition~\cite{Lovett2015} of Sanders's quasipolynomial Bogolyubov theorem.

In a spectacular breakthrough posted to the arXiv right before the due date for this
survey article, Gowers, Green, Manners, and Tao~\cite{GowersGreenMannersTao2023} have
proved Conjecture~\ref{conj:PFR} in the case $p=2$ and announced a forthcoming proof of
the conjecture for all odd primes as well. Further developing the theory of sumsets and
entropy studied in~\cite{GreenMannersTao2023},~\cite{Ruzsa2009}, and~\cite{Tao2010}, their
argument proceeds by proving an entropic version of the polynomial Freiman--Ruzsa
conjecture, which the latter three authors showed is equivalent to
Conjecture~\ref{conj:PFR} in (also very recent) earlier
work~\cite{GreenMannersTao2023}. Time and space constraints unfortunately prevent us from
giving an exposition of the elegant argument of Gowers, Green, Manners, and Tao in this
survey, so the reader is encouraged to read their very well-written and almost
self-contained preprint~\cite{GowersGreenMannersTao2023}.

Though the polynomial Freiman--Ruzsa conjecture is now settled, the argument
in~\cite{GowersGreenMannersTao2023} does not seem to adapt to prove a polynomial
Bogolyubov theorem, so the question of whether such a result holds is still open.

\begin{qn}
  Fix a prime $p$. Is it the case that there exists a constant $C>0$ such that, whenever
  $A\subset\mathbf{F}_p^n$ has density $\alpha$, $2A-2A$ contains a subspace of
  codimension at most $C\log{\alpha^{-1}}$?
\end{qn}

\subsection{Inverse theorems for higher degree uniformity norms}

Bergelson, Tao, and Ziegler~\cite{BergelsonTaoZiegler10,TaoZiegler10} proved inverse
theorems for the $U^s$-norm in the finite field model setting for all $s\geq 4$, provided
that $p$ is sufficiently large in terms of $s$.
\begin{theorem}\label{thm:Usff}
  Let $s$ be a natural number and $p\geq s$. If $f:\mathbf{F}_p^n\to\mathbf{C}$ is
  $1$-bounded and $\|f\|_{U^s}\geq\delta$, then there exists a polynomial
  $P:\mathbf{F}_p^n\to\mathbf{F}_p$ of degree at most $s-1$ such that
  \[
    \left|\mathbf{E}_{x\in\mathbf{F}_p^n}f(x)e_p(P(x))\right|\gg_{\delta,s,p}1.
  \]
\end{theorem}
The proof of Theorem~\ref{thm:Usff} proceeds via ergodic theory, and thus produces no
quantitative bounds for the correlation of $f$ with a polynomial phase.

While Theorem~\ref{thm:U3inverse} holds for all primes $p$, it turns out that
Theorem~\ref{thm:Usff} is false for $p<s$ as soon as $s\geq 4$, as was observed by Green
and Tao~\cite{GreenTao09II} and Lovett, Meshulam, and
Samorodnitsky~\cite{LovettMeshulamSamorodnitsky11}. Despite this, Tao and
Ziegler~\cite{TaoZiegler12} showed that Theorem~\ref{thm:Usff} can be modified to a
statement that holds for all primes by enlarging the class of polynomial phases to include
those coming from \defword{non-classical polynomials} of degree at most $s-1$, i.e.,
functions $P:\mathbf{F}_p^n\to\mathbf{T}$ satisfying
\[
\partial_{h_1,\dots,h_s}P(x)=0
\]
for all $x,h_1,\dots,h_s\in\mathbf{F}_p^n$. 
\begin{theorem}\label{thm:Usfflowchar}
  Let $s$ be a natural number. If $f:\mathbf{F}_p^n\to\mathbf{C}$ is $1$-bounded and
  $\|f\|_{U^s}\geq\delta$, then there exists a non-classical polynomial
  $P:\mathbf{F}_p\to\mathbf{T}$ of degree at most $s-1$ such that
  \[
    \left|\mathbf{E}_{x\in\mathbf{F}_p^n}f(x)e(P(x))\right|\gg_{\delta,s,p}1.
  \]
\end{theorem}
Theorem~\ref{thm:Usff} says that when $p\geq s$, the $P$ in Theorem~\ref{thm:Usfflowchar}
can be assumed to take values in $\frac{1}{p}\mathbf{Z}\pmod{1}$.

At the time of Wolf's survey, it was a major open problem to prove quantitative versions
of the inverse theorem for the $U^s$-norms when $s\geq 4$. Since then, Gowers and
Mili\'cevi\'c~\cite{GowersMilicevic17,GowersMilicevic20} have proved a
quantitative version of Theorem~\ref{thm:Usff} with reasonable bounds.
\begin{theorem}\label{thm:GM}
  Let $s$ be a natural number and $p\geq s$. There exists a natural number $m=m(s)$ and a
  constant $c=c(s,p)>0$ such that the following holds. If $f:\mathbf{F}_p^n\to\mathbf{C}$
  is $1$-bounded and $\|f\|_{U^s}\geq\delta$, then there exists a polynomial
  $P:\mathbf{F}_p^n\to\mathbf{F}_p$ of degree at most $s-1$ such that
  \[
    \left|\mathbf{E}_{x\in\mathbf{F}_p^n}f(x)e_p(P(x))\right|\gg_{s,p}\frac{1}{\exp^{(m)}\left(c\delta^{-1}\right)}.
  \]
\end{theorem}
Quantitative bounds in the low characteristic case are only known for $s\leq 6$, due to
work of Tidor~\cite{Tidor22} and Mili\'cevi\'c~\cite{Milicevic22}, so the following
problem is still open.
\begin{problem}
  Prove a version of Theorem~\ref{thm:Usfflowchar} with reasonable quantitative bounds
  when $s>6$.
\end{problem}

The value of $m$ obtained in Theorem~\ref{thm:GM} by Gowers and
Mili\'cevi\'c's~\cite{GowersMilicevic20} argument grows like $3^ss!$, in contrast to the
inverse theorems of Manners~\cite{Manners18}, which are of the same quality for all $s\geq 4$. This prompts the
following natural problem.
\begin{problem}
  Improve the bounds in Theorem~\ref{thm:GM}.
\end{problem}
Any progress on this problem would automatically improve any result, such as the main
theorem in~\cite{Peluse24}, that depends on Theorem~\ref{thm:GM}. It would be very
interesting to see a proof of Theorem~\ref{thm:GM} yielding a constant tower height
$m\ll 1$, independent of $s$. Even obtaining $m\ll s$ will require several new ideas. Kim,
Li, and Tidor~\cite{KimLiTidor23} have obtained improved bounds in the inverse theorem for
the $U^4$-norm in all characteristics, showing that
\[
\left|\mathbf{E}_{x\in\mathbf{F}_p^n}f(x)e_p(P(x))\right|\gg\frac{1}{\exp^{(2)}\left(c_p\log(2/\delta)^{c_p}\right)}
\]
for some constants $c_p>0$ depending only on $p$. Considering that quasipolynomial bounds
in the bilinear Bogolyubov theorem, a key ingredient in the proof of the $U^4$-inverse
theorem in the finite field model setting, are known thanks to work of Hosseini and
Lovett~\cite{HosseiniLovett19}, quasipolynomial bounds in the $U^4$-inverse theorem should
be within reach, though more will have to be done to improve the quantitative aspects of
other parts of the argument. Given the resolution of the polynomial Freiman--Ruzsa
conjecture, it could be that polynomial bounds are even within reach.

\section{The polynomial Szemer\'edi theorem}
In 1977, Furstenberg~\cite{Furstenberg77} gave an alternative proof of Szemer\'edi's
theorem via ergodic theory, in which he introduced his now-famous correspondence
principle and created the field of ergodic Ramsey theory. Furstenberg's argument has now
been extended to prove very broad generalizations of Szemer\'edi's theorem, most notably a
multidimensional generalization due to Furstenberg and
Katznelson~\cite{FurstenbergKatznelson78} and a polynomial generalization due to Bergelson
and Leibman~\cite{BergelsonLeibman96}.
\begin{theorem}[Multidimensional Szemer\'edi Theorem]\label{thm:multidim}
  Let $S\subset\mathbf{Z}^d$ be finite and nonempty. If $A\subset[N]^d$ contains no
  nontrivial homothetic copies
  \begin{equation*}
    a+b\cdot S\qquad(b\neq 0)
  \end{equation*}
  of $S$, then $|A|=o_S(N^d)$.
\end{theorem}
There are now multiple proofs of Theorem~\ref{thm:multidim} using hypergraph regularity
methods~\cite{Gowers07,NagleRodlSchacht06,RodlSkokan04,Tao06}, which all give a savings
over the trivial bound $|A|\leq N^d$ of inverse Ackermann-type.
\begin{theorem}[Polynomial Szemer\'edi Theorem]\label{thm:poly}
  Let $P_1,\dots,P_m\in\mathbf{Z}[y]$, all satisfying $P_i(0)=0$. If $A\subset[N]$
  contains no nontrivial polynomial progressions
  \begin{equation}\label{eq:polyprog}
    x,x+P_1(y),\dots,x+P_m(y)\qquad(y\neq 0),
  \end{equation}
  then $|A|=o_{P_1,\dots,P_m}(N)$.
\end{theorem}
The assumption that $P_i(0)=0$ for $i=1,\dots,m$ in Theorem~\ref{thm:poly} is there to
prevent local obstructions to the result being true. For example, note that the even integers contain
no configurations of the form $x,x+2y+1$, since $2y+1$ is always odd when $y$ is an
integer. The only known proofs of Theorem~\ref{thm:poly} in full generality are via
ergodic theory, and give no quantitative bounds at all.

Following his proof of reasonable bounds in Szemer\'edi's
theorem, Gowers~\cite{Gowers01S} posed the problem of proving reasonable bounds in
Theorems~\ref{thm:multidim} and~\ref{thm:poly}.
\begin{problem}\label{prob:quant}
  Prove quantitative versions of the multidimensional and polynomial generalizations of
  Szemer\'edi's theorem, with reasonable bounds.
\end{problem}
This problem has turned out to be very difficult, though a good amount of progress has
finally been made in the past few years on proving a quantitative version of the
polynomial Szemer\'edi theorem. We will devote the remainder of this section to discussing
this progress, and say a bit about the multidimensional Szemer\'edi theorem in Section~\ref{sec:additional}.

Prior to a couple of years ago, quantitative bounds were known in Theorem~\ref{thm:poly}
in only three special cases:
\begin{enumerate}
\item when $m=1$, as discussed in Subsection~\ref{ssec:further},
\item when all of the $P_i$ are linear, which follows from Gowers's quantitative proof of Szemer\'edi's theorem,
\item and when the $P_i$ are all monomials of the same fixed degree, due to work of
  Prendiville~\cite{Prendiville17}.
\end{enumerate}
There are insurmountable obstructions to applying the methods used to handle the above
three special cases to any additional polynomial progressions, such as the
\textit{nonlinear Roth configuration},
\begin{equation*}
  x,x+y,x+y^2,
\end{equation*}
which is the simplest nonlinear polynomial progression of length greater than $2$. Indeed,
S\'ark\"ozy's argument only works for two-term polynomial progressions, as its starting
point is the fact that the count of such progressions in a set can be written as the inner product
of two functions, one of which is a convolution. Gowers's argument crucially relies on the
fact that arithmetic progressions are translation- and dilation-invariant, and Prendiville was
able to generalize Gowers's proof by using the fact that arithmetic progressions with common
difference equal to a perfect $d^{th}$ power are invariant under dilations by a perfect
$d^{th}$ power. No other polynomial progressions (including the nonlinear Roth
configuration) are anywhere close to being dilation-invariant, so the configurations
considered by Prendiville in~\cite{Prendiville17} are exactly those that can be handled
using Gowers's methods.

For every $P_1,\dots,P_m\in\Z[y]$ and $S$ equal to either $[N]$ or $\mathbf{F}_p$, let
$r_{P_1,\dots,P_m}(S)$ denote the size of the largest subset of $S$ lacking the nontrivial
progression~\eqref{eq:polyprog}. Observe that
$r_{P_1,\dots,P_m}(\mathbf{F}_p)\leq r_{P_1,\dots,P_m}([p])$, since reducing a nontrivial
polynomial progression in $[p]$ modulo $p$ produces a nontrivial polynomial progression in
$\mathbf{F}_p$. Thus, any bounds obtained in the integer setting automatically give bounds
in the finite field setting. As was discussed in Subsection~\ref{ssec:further}, the finite
field setting is strictly easier than the integer setting when $m=1$. This is true also
for longer polynomial progressions, as evidenced by the following result of Bourgain and
Chang~\cite{BourgainChang17} giving power-saving bounds for subsets of finite fields
lacking the nonlinear Roth configuration.
\begin{theorem}
  We have
  \[
    r_{y,y^2}(\mathbf{F}_p)\ll p^{14/15}.
  \]
\end{theorem}
The argument of Bourgain and Chang was very specific to the progression $x,x+y,x+y^2$, as
it used the fact that one can explicitly evaluate quadratic Gauss sums. Their method could only
possibly generalize to progressions involving one linear polynomial and one quadratic
polynomial, so they asked whether a similar power-saving bound holds in general for
three-term polynomial progression involving linearly independent polynomials. Using a
different approach, I answered their question in the affirmative, with a slightly worse
exponent~\cite{Peluse18}.
\begin{theorem}\label{thm:3pp}
  Let $P_1,P_2\in\Z[y]$ be linearly independent and satisfy $P_1(0)=P_2(0)=0$. Then
  $r_{P_1,P_2}(\mathbf{F}_p)\ll_{P_1,P_2}p^{23/24}$.
\end{theorem}
Dong, Li, and Sawin~\cite{DongLiSawin20} very shortly after improved the exponent in
Theorem~\ref{thm:3pp}, obtaining the bound $r_{P_1,P_2}(\mathbf{F}_p)\ll_{P_1,P_2} p^{11/12}$.

None of the arguments in~\cite{BourgainChang17},~\cite{DongLiSawin20}, or~\cite{Peluse18}
seem to adapt to prove quantitative bounds for subsets of finite fields lacking longer
polynomial progressions or to the integer setting. Because of this, I introduced a new
technique, now known as ``degree-lowering'', and used it to prove power-saving bounds for
subsets of finite fields lacking arbitrarily long polynomial progressions, provided the
polynomials are linearly independent~\cite{Peluse19}.
\begin{theorem}
  Let $P_1,\dots,P_m\in\Z[y]$ be linearly independent polynomials, all satisfying
  $P_i(0)=0$. There exists $\gamma_{P_1,\dots,P_m}>0$ such that
  $r_{P_1,\dots,P_m}(\mathbf{F}_p)\ll p^{1-\gamma_{P_1,\dots,P_m}}$.
\end{theorem}
The degree-lowering technique is robust enough that it works in a wide variety of
settings, including the integer, continuous, and ergodic settings. In the integer
setting, Prendiville and I used it to prove quantitative bounds for subsets of $[N]$
lacking the nonlinear Roth configuration~\cite{PelusePrendiville19,PelusePrendiville20},
and I then generalized our argument to handle arbitrarily long polynomial progressions
with distinct degrees~\cite{Peluse20}. Kuca~\cite{Kuca21II,Kuca21III} and
Leng~\cite{Leng22} have also used the degree-lowering method to prove good quantitative
bounds in more special cases of the polynomial Szemer\'edi theorem in finite fields. More
applications relevant to the topic of this survey will be discussed in the last subsection
of this section.

\subsection{Degree-lowering and the non-linear Roth configuration}\label{ssec:deglow}

We will now give an illustration of the degree-lowering method by using it to prove a
power-saving bound for sets lacking the nonlinear Roth configuration. The following
argument is an adaptation of the proof in~\cite{PelusePrendiville19} to the finite field
setting, which greatly simplifies it.

Fix a prime $p>2$ and define, for all $f_0,f_1,f_2:\mathbf{F}_p\to\mathbf{C}$, the
trilinear average
\begin{equation*}
  \Lambda(f_0,f_1,f_2):=\mathbf{E}_{x,y\in\mathbf{F}_p}f_0(x)f_1(x+y)f_2(x+y^2).
\end{equation*}
\begin{lemma}\label{lem:pet}
  Let $f_0,f_1,f_2:\mathbf{F}_p\to\mathbf{C}$ be $1$-bounded. Then
  \begin{equation*}
    \left|\Lambda(f_0,f_1,f_2)\right|^4\ll\|f_2\|_{U^3}+\frac{1}{p}.
  \end{equation*}
\end{lemma}
\begin{proof}
First, write
  \begin{equation*}
    \Lambda(f_0,f_1,f_2)=\mathbf{E}_{x\in\mathbf{F}_p}f_0(x)\left(\mathbf{E}_{y\in\mathbf{F}_p}f_1(x+y)f_2(x+y^2)\right),
  \end{equation*}
so that, by the Cauchy--Schwarz inequality,
\begin{equation*}
  \left|\Lambda(f_0,f_1,f_2)\right|^2\leq\mathbf{E}_{x,y,a\in\mathbf{F}_p}f_1(x+y)\overline{f_1(x+y+a)}f_2(x+y^2)\overline{f_2(x+(y+a)^2)},
\end{equation*}
since $f_0$ is $1$-bounded. Making the change of variables $x\mapsto x-y$, we can write
the right-hand side of the above as
\begin{equation*}
  \mathbf{E}_{x,y,a\in\mathbf{F}_p}f_1(x)\overline{f_1(x+a)}f_2(x+y^2-y)\overline{f_2(x+(y+a)^2-y)},
\end{equation*}
and then apply the Cauchy--Schwarz inequality again and make the change of variables
$x\mapsto x+y$ to get that $\left|\Lambda(f_0,f_1,f_2)\right|^4$ is bounded above by
\begin{equation*}
  \mathbf{E}_{x,y,a,b\in\mathbf{F}_p}f_2(x+y^2)\overline{f_2(x+(y+a)^2)}\overline{f_2(x+(y+b)^2-b)}f_2(x+(y+a+b)^2-b).
\end{equation*}
Making the change of variables $x\mapsto x-y^2$ rewrites
the above as
\begin{equation}\label{eq:linearization}
  \mathbf{E}_{a,b\in\mathbf{F}_p}\mathbf{E}_{x,y\in\mathbf{F}_p}f_2(x)\overline{g_{1,a,b}(x+2ay)}\overline{g_{2,a,b}(x+2by)}g_{3,a,b}(x+2(a+b)y),
\end{equation}
where, for each $i=1,2,3$, $g_{i,a,b}(x)=f_2(x+\phi_i(a,b))$ for some function
$\phi_i:\mathbf{F}_p^2\to\mathbf{F}_p$.

It is a standard fact, which we will soon prove, that the inner average
in~\eqref{eq:linearization} is controlled by the $U^3$-norm of $f_2$ whenever $a,b,$
and $a+b$ are nonzero. But, first observe that
\begin{equation*}
  \#\left\{(a,b)\in\mathbf{F}_p^2:a=0,\ b=0,\ \text{or }a+b=0\right\}=3p-2,
\end{equation*}
so that the total contribution to~\eqref{eq:linearization} coming from pairs $(a,b)$ such
that $a,b,$ and $a+b$ are nonzero is at most $\frac{3p-2}{p^2}$. Thus,~\eqref{eq:linearization} is
\begin{equation}\label{eq:linearizationbound}
  \ll\mathbf{E}'_{a,b\in\mathbf{F}_p}\left|\mathbf{E}_{x,y\in\mathbf{F}_p}f_2(x)\overline{g_{1,a,b}(x+2ay)}\overline{g_{2,a,b}(x+2by)}g_{3,a,b}(x+2(a+b)y)\right|+\frac{1}{p},
\end{equation}
where $\mathbf{E}'_{a,b\in\mathbf{F}_p}$ denotes the average over pairs $(a,b)$ for which
$a,b,$ and $a+b$ are nonzero.

Now, we will bound
\begin{equation}\label{eq:linearizationfix}
 \left|\mathbf{E}_{x,y\in\mathbf{F}_p}f_2(x)\overline{g_{1,a,b}(x+2ay)}\overline{g_{2,a,b}(x+2by)}g_{3,a,b}(x+2(a+b)y)\right|
\end{equation}
whenever $a,b,$ and $a+b$ are nonzero. By another application of the Cauchy--Schwarz
inequality, the square
of~\eqref{eq:linearizationfix} is at most
\begin{equation*}
  \mathbf{E}_{x,y,h_1\in\mathbf{F}_p}\Delta_{2ah_1}\overline{g_{1,a,b}(x+2ay)}\Delta_{2bh_1}\overline{g_{2,a,b}(x+2by)}\Delta_{2(a+b)h_1}g_{3,a,b}(x+2(a+b)y),
\end{equation*}
which equals
\begin{equation*}
  \mathbf{E}_{x,y,h_1\in\mathbf{F}_p}\Delta_{2ah_1}\overline{g_{1,a,b}(x)}\Delta_{2bh_1}\overline{g_{2,a,b}(x+2(b-a)y)}\Delta_{2(a+b)h_1}g_{3,a,b}(x+2by),
\end{equation*}
by making the change of variables $x\mapsto x-2ay$. We apply the Cauchy--Schwarz
inequality again to bound the square of the above by
\begin{equation*}
  \mathbf{E}_{x,y,h_1,h_2\in\mathbf{F}_p}\Delta_{2bh_1,2(b-a)h_2}\overline{g_{2,a,b}(x+2(b-a)y)}\Delta_{2(a+b)h_1,2bh_2}g_{3,a,b}(x+2by),
\end{equation*}
which, by the change of variables $x\mapsto x-2(b-a)y$, equals
\begin{equation*}
  \mathbf{E}_{x,y,h_1,h_2\in\mathbf{F}_p}\Delta_{2bh_1,2(b-a)h_2}\overline{g_{2,a,b}(x)}\Delta_{2(a+b)h_1,2bh_2}g_{3,a,b}(x+2ay).
\end{equation*}
A final application of the Cauchy--Schwarz inequality and a change of variables bounds the
square of the above by
  \begin{equation*}
  \mathbf{E}_{x,h_1,h_2,h_3\in\mathbf{F}_p}\Delta_{2(a+b)h_1,2bh_2,2ah_3}g_{3,a,b}(x).
\end{equation*}
Since $p>2$ and $a,b,$ and $a+b$ are all nonzero, making the change of variables
$h_1\mapsto \frac{h_1}{2(a+b)}$, $h_2\mapsto\frac{h_2}{2b}$, $h_3\mapsto\frac{h_3}{2a}$,
and $x\mapsto x-\phi_3(a,b)$ reveals that
\begin{equation*}
  \mathbf{E}_{x,h_1,h_2,h_3\in\mathbf{F}_p}\Delta_{2(a+b)h_1,2bh_2,2ah_3}g_{3,a,b}(x)=\|f_2\|_{U^3}^8.
\end{equation*}
Combining this with~\eqref{eq:linearizationbound} yields
\begin{equation*}
  \left|\Lambda(f_0,f_1,f_2)\right|^4\ll \|f_2\|_{U^3}+\frac{1}{p}.
\end{equation*}
\end{proof}

A simple application of the Cauchy--Schwarz inequality shows that not only is
$|\Lambda(f_0,f_1,f_2)|$ bounded by the $U^3$-norm of $f_2$, but it is also bounded by the
$U^3$-norm of a \textit{dual function}.

\begin{lemma}\label{lem:dual}
  Let $f_0,f_1,f_2:\mathbf{F}_p\to\mathbf{C}$ be $1$-bounded, and define the dual function
  $F_2:\mathbf{F}_p\to\mathbf{C}$ by
  \begin{equation*}
    F_2(x):=\mathbf{E}_{z\in\mathbf{F}_p}f_0(x-z^2)f_1(x-z^2+z).
  \end{equation*}
  Then
  \begin{equation*}
    \left|\Lambda(f_0,f_1,f_2)\right|^8\ll\|F_2\|_{U^3}+\frac{1}{p}.
  \end{equation*}
\end{lemma}
\begin{proof}
By making the change of variables $x\mapsto x-y^2$, we can write
\begin{equation*}
  \Lambda(f_0,f_1,f_2)=\mathbf{E}_{x,y\in\mathbf{F}_p}f_0(x-y^2)f_1(x+y-y^2)f_2(x),
\end{equation*}
so that, by the Cauchy--Schwarz inequality and the change of variables $x\mapsto x+y^2$,
\begin{align*}
  \left|\Lambda(f_0,f_1,f_2)\right|^2&\leq\mathbf{E}_{x,y,z\in\mathbf{F}_p}f_0(x-z^2)f_1(x+z-z^2)\overline{f_0(x-y^2)}\overline{f_1(x+y-y^2)}
  \\
                                     &=\mathbf{E}_{x,y,z\in\mathbf{F}_p}\overline{f_0(x)}\overline{f_1(x+y)}F_2(x+y^2) \\
  &=\Lambda\left(\overline{f_0},\overline{f_1},F_2\right).
\end{align*}
Noting that $F_2$ is $1$-bounded (being the average of $1$-bounded functions), the desired
bound now follows from applying Lemma~\ref{lem:pet}.
\end{proof}

The key insight of~\cite{Peluse19} is that the $U^s$-norm of a polynomial dual function
involving linearly independent polynomials can be bounded in terms of the $U^{s-1}$-norm
of the dual function. The following lemma presents the simplest nontrivial instance of this phenomenon. 

\begin{theorem}[One step degree-lowering]\label{thm:deglow}
  Let $f,g:\mathbf{F}_p\to\mathbf{C}$ be $1$-bounded, and set
  \begin{equation*}
    F(x):=\mathbf{E}_{z\in\mathbf{F}_p}f(x-z^2)g(x+z-z^2).
  \end{equation*}
  If $\|F\|_{U^3}^8\geq 4/p$, then
  \begin{equation*}
    \left\|F\right\|_{U^3}^8\ll\left\| F\right\|_{U^2}.
  \end{equation*}
\end{theorem}

Theorem~\ref{thm:deglow} is proved by combining two lemmas.

\begin{lemma}[Difference-dual interchange]\label{lem:dualdifference}
  Let $K\subset\mathbf{F}_p$, $f,g:\mathbf{F}_p\to\mathbf{C}$ be $1$-bounded, and set
  \begin{equation*}
    F(x):=\mathbf{E}_{z\in\mathbf{F}_p}f(x-z^2)g(x+z-z^2).
  \end{equation*}
  Then, for every function $\phi:\mathbf{F}_p\to\mathbf{F}_p$, we have that
  \begin{equation*}
    \mathbf{E}_{k\in K}\left|\mathbf{E}_{x}\Delta_{k}F(x)e_p(\phi(k) x)\right|^2  
  \end{equation*}
  is bounded above by
  \begin{equation*}
    \mathbf{E}_{k,k'\in K}\left|\mathbf{E}_{x,y\in\mathbf{F}_p}\Delta_{k'-k}f(x)\Delta_{k'-k}g(x+y)e_p([\phi(k)-\phi(k')] [x+y^2])\right|^2.
  \end{equation*}
\end{lemma}
\begin{proof}
Expanding the definition of the dual function $F$, we have that
\begin{equation*}
 \mathbf{E}_{k\in K}\left|\mathbf{E}_{x\in\mathbf{F}_p}\Delta_{k}F(x)e_p(\phi(k) x)\right|^2
\end{equation*}
equals
\begin{align*}
 \mathbf{E}_{k\in
  K}\mathbf{E}_{x,x',z_1,z_2,z_3,z_4\in\mathbf{F}_p}\Big[&f(x-z_1^2)g(x+z_1-z_1^2)\\
  &\overline{f(x-z_2^2+k)g(x+z_2-z_2^2+k)}e_p(\phi(k)
 x) \\
                                                      &\overline{f(x'-z_3^2)g(x'+z_3-z_3^2)}\\
  &f(x'-z_4^2+k)g(x'+z_4-z_4^2+k)e_p(-\phi(k)
 x')\Big],
\end{align*}
which can be written as
\begin{align*}
  \mathbf{E}_{x,x',z_1,z_2,z_3,z_4\in\mathbf{F}_p}\Big[&f(x-z_1^2)g(x+z_1-z_1^2)\overline{f(x'-z_3^2)g(x'+z_3-z_3^2)} \\
 & \Big(\mathbf{E}_{k\in
   K}\overline{f(x-z_2^2+k)g(x+z_2-z_2^2+k)} \\
  &\ \ \ \ \ \ \ \ \ f(x'-z_4^2+k)  g(x'+z_4-z_4^2+k)e_p(\phi(k)
 [x-x'])\Big)\Big],
\end{align*}
Thus, by the Cauchy--Schwarz inequality,
$\mathbf{E}_{k\in K}\left|\mathbf{E}_{x\in\mathbf{F}_p}\Delta_{k}F(x)e_p(\phi(k)
  x)\right|^2$ is bounded above by
\begin{equation*}
\mathbf{E}_{k,k'\in K} \left|\mathbf{E}_{x,z\in\mathbf{F}_p}\Delta_{k'-k}f(x-z^2+k)\Delta_{k'-k}g(x+z-z^2+k)e_p([\phi(k)-\phi(k')] x)\right|^2,
\end{equation*}
which, by the change of variables $x\mapsto x-k+z^2$, is bounded above by
\begin{equation*}
 \mathbf{E}_{k,k'\in K}\left|\mathbf{E}_{x,z\in\mathbf{F}_p}\Delta_{k'-k}f(x)\Delta_{k'-k}g(x+z)e_p([\phi(k)-\phi(k')] [x+z^2])\right|^2.
\end{equation*}
\end{proof}

\begin{lemma}[Major arc lemma]\label{lem:majorarc}
  Let $f,g:\mathbf{F}_p\to\mathbf{C}$ be $1$-bounded. If $\xi\in\mathbf{F}_p$ is nonzero,
  then
  \begin{equation*}
    \left|\mathbf{E}_{x,y\in\mathbf{F}_p}f(x)g(x+y)e_p(\xi[x+y^2])\right|\leq\frac{1}{\sqrt{p}}.
  \end{equation*}
\end{lemma}
\begin{proof}
  Plugging in the Fourier inversion formula for $f$ and $g$ and using orthogonality of characters yields
  \begin{align*}
    \mathbf{E}_{x,y\in\mathbf{F}_p}f(x)g(x+y)e_p(\xi[x+y^2]) &=
                                                               \sum_{\zeta,\eta\in\mathbf{F}_p}\widehat{f}(\zeta)\widehat{g}(\eta)\mathbf{E}_{x,y\in\mathbf{F}_p}e_p([\zeta+\eta+\xi]x+\eta
                                                               y+\xi y^2) \\
    &=
                                                               \sum_{\eta\in\mathbf{F}_p}\widehat{f}(-\eta-\xi)\widehat{g}(\eta)\mathbf{E}_{y\in\mathbf{F}_p}e_p(\eta
                                                               y+\xi y^2).
  \end{align*}
  Thus, since $\left|\mathbf{E}_{y\in\mathbf{F}_p}e_p(\eta y+\xi y^2)\right|\leq p^{-1/2}$
  whenever $\xi\neq 0$, we have
\begin{equation*}
  \left|\mathbf{E}_{x,y\in\mathbf{F}_p}f(x)g(x+y)e_p(\xi[x+y^2])\right|\leq\frac{1}{\sqrt{p}}\sum_{\eta\in\mathbf{F}_p}|\widehat{f}(-\eta-\xi)||\widehat{g}(\eta)|\leq\frac{1}{\sqrt{p}}\left\|f\right\|_{L^2}\left\|g\right\|_{L^2},
\end{equation*}
and the conclusion of the lemma follows from the $1$-boundedness of $f$ and $g$.
\end{proof}

Now we can prove Theorem~\ref{thm:deglow}.
\begin{proof}
By the $U^2$-inverse theorem, we have
\begin{equation*}
  \left\|F\right\|^8_{U^3} = \mathbf{E}_{k\in\mathbf{F}_p}\|\Delta_kF\|_{U^2}^4\leq \mathbf{E}_{k\in\mathbf{F}_p}\|\widehat{\Delta_kF}\|_{\ell^\infty}^2.
\end{equation*}
Set
\begin{equation*}
  K:=\left\{k\in\mathbf{F}_p:\|\widehat{\Delta_kF}\|^2_{\ell^\infty}\geq \|F\|^8_{U^3}/2\right\},
\end{equation*}
so that $|K|/p\geq \|F\|^8_{U^3}/2$ and
\begin{equation*}
  \frac{\|F\|^8_{U^3}}{2}\leq \mathbf{E}_{k\in K}\|\widehat{\Delta_kF}\|_{\ell^\infty}^2
\end{equation*}
Let $\phi:\mathbf{F}_p\to\mathbf{F}_p$ be such that
$\left|\widehat{\Delta_kF}(\phi(k))\right|=\|\widehat{\Delta_kF}\|_{\ell^\infty}$ for all
$k\in\mathbf{F}_p$. Then Lemma~\ref{lem:dualdifference} says that
$\mathbf{E}_{k\in K}\|\widehat{\Delta_kF}\|_{\ell^\infty}^2$ is bounded above by
\begin{equation}\label{eq:dualdifference}
  \mathbf{E}_{k,k'\in K}\left|\mathbf{E}_{x,y\in\mathbf{F}_p}\Delta_{k'-k}f(x)\Delta_{k'-k}g(x+y)e_p([\phi(k)-\phi(k')] [x+y^2])\right|^2.
\end{equation}
Applying Lemma~\ref{lem:majorarc} to~\eqref{eq:dualdifference} therefore yields
\begin{equation*}
  \frac{\left\|F\right\|^8_{U^3}}{2}\leq \frac{\#\left\{(k,k')\in K^2:\phi(k)=\phi(k')\right\}}{|K|^2}+\frac{1}{p},
\end{equation*}
which implies that
\begin{equation*}
  \frac{\|F\|^8_{U^3}}{4}\leq \frac{\#\left\{(k,k')\in K^2:\phi(k)=\phi(k')\right\}}{|K|^2},
\end{equation*}
by our lower bound assumption on $\|F\|_{U^3}$. By the pigeonhole principle, there must
exist a $k'\in K$ for which
\begin{equation*}
  \frac{\#\left\{k\in K:\phi(k)=\phi(k')\right\}}{|K|}\geq \frac{\|F\|^8_{U^3}}{4}.
\end{equation*}
Fix this $k'$, and set $\beta:=\phi(k')$. Then, since $|K|/p\geq \|F\|^8_{U^3}/2$, the
above implies that
\begin{equation*}
  \frac{\#\{k\in\mathbf{F}_p:\phi(k)=\beta\}}{p}\geq \frac{\|F\|^8_{U^3}}{8}.
\end{equation*}
Thus,
\begin{equation*}
  \mathbf{E}_{k\in\mathbf{F}_p}\left|\mathbf{E}_{x\in\mathbf{F}_p}\Delta_kF(x)e_p(\beta
    x)\right|^2\geq
  \frac{\|F\|^8_{U^3}}{8}\mathbf{E}_{k\in\mathbf{F}_p}\left\|\widehat{\Delta_kF}\right\|^2_{\ell^\infty}\geq \frac{\|F\|^{16}_{U^3}}{8}.
\end{equation*}

Now, expanding the square on the left-hand side of the above,
\begin{equation*}
  \mathbf{E}_{k\in\mathbf{F}_p}\left|\mathbf{E}_{x\in\mathbf{F}_p}\Delta_kF(x)e_p(\beta
    x)\right|^2=\mathbf{E}_{x,h,k\in\mathbf{F}_p}\Delta_{h}F(x)\overline{\Delta_{h}F(x+k)}e_p(-\beta h),
\end{equation*}
so that, by the Cauchy--Schwarz inequality,
\begin{equation*}
  \mathbf{E}_{k\in\mathbf{F}_p}\left|\mathbf{E}_{x\in\mathbf{F}_p}\Delta_kF(x)e_p(\beta
    x)\right|^2\leq \|F\|_{U^2}^2.
\end{equation*}
Putting everything together then gives
\begin{equation*}
\frac{\|F\|^{16}_{U^3}}{8}  \leq \|F\|^2_{U^2}.
\end{equation*}
\end{proof}

Combining control of $\Lambda(f_0,f_1,f_2)$ by $\|F_2\|_{U^3}$ with the degree-lowering
lemma quickly yields not only a power-saving bound on sets lacking the nonlinear Roth
configuration, but that any subset of $\mathbf{F}_p$ of sufficiently large density
contains very close to the number of non-linear Roth configurations one would expect in a
random set of the same density.

\begin{theorem}
  If $A\subset\mathbf{F}_p$ has density $\alpha$, then
  \begin{equation*}
    \#\left\{(x,y)\in\mathbf{F}_p^2:x,x+y,x+y^2\in A\right\}=\alpha^3 p^2+O\left(p^{2-1/144}\right).
  \end{equation*}
  Thus, if $\alpha>p^{-1/500}$, then $A$ contains a nontrivial nonlinear Roth
  configuration.
\end{theorem}
This recovers the main result of Bourgain and Chang from~\cite{BourgainChang17}, but with
a weaker exponent.
\begin{proof}
First observe that
\begin{align*}
  \Lambda(1,1_A,1_A) &= \mathbf{E}_{x,y\in\mathbf{F}_p}1_A(x)1_A(x+y^2-y) \\
  &= \sum_{\xi\in\mathbf{F}_p}\left|\widehat{1_A}(\xi)\right|^2\left(\mathbf{E}_{y\in\mathbf{F}_p}e_p(\xi[y^2-y])\right) \\
  &= \alpha^2+O\left(\frac{1}{\sqrt{p}}\right)
\end{align*}
by Fourier inversion, orthogonality of characters, and the fact that
$|\mathbf{E}_{y\in\mathbf{F}_p}e_p(\xi[y^2-y])|\leq p^{-1/2}$ whenever $\xi\neq
0$. Setting $f_A:=1_A-\alpha$, we therefore have
\begin{equation*}
  \left|\Lambda(1_A,1_A,1_A)-\alpha^3\right|\leq\left|\Lambda(f_A,1_A,1_A)\right|+O\left(\frac{1}{\sqrt{p}}\right).
\end{equation*}
By Lemmas~\ref{lem:dual} and~\ref{thm:deglow},
\begin{equation*}
      \left|\Lambda(f_0,f_1,f_2)\right|^8\ll\|F_2\|^{1/8}_{U^2}+\frac{1}{p}.
\end{equation*}
where $F_2(x):=\mathbf{E}_{z\in\mathbf{F}_p}f_A(x-z^2)1_A(x+z-z^2)$. But, by the
$U^2$-inverse theorem, there exists a $\xi\in\mathbf{F}_p$ such that
\begin{equation*}
  \|F_2\|_{U^2}^2\leq \left|\mathbf{E}_{x\in\mathbf{F}_p}F_2(x)e_p(\xi x)\right|=\left|\mathbf{E}_{x,z\in\mathbf{F}_p}f_A(x)1_A(x+z)e_p(\xi[x+z^2])\right|.
\end{equation*}
If $\xi\neq 0$, then $\|F_2\|_{U^2}^2\leq p^{-1/2}$ by Lemma~\ref{lem:majorarc}, and if
$\xi=0$, then $\|F_2\|^2_{U^2}=0$ because $f_A$ has mean zero. So, in either case, we have
\begin{equation*}
  \left|\Lambda(f_0,f_1,f_2)\right|\ll\frac{1}{p^{1/144}},
\end{equation*}
from which we conclude that
\begin{equation*}
  \left|\Lambda(1_A,1_A,1_A)-\alpha^3\right|\ll\frac{1}{p^{1/144}}.
\end{equation*}
\end{proof}

\subsection{Further results and questions}
Bourgain and Chang~\cite{BourgainChang17} also asked whether power-saving bounds hold for
subsets of finite fields lacking three-term progressions involving linearly independent
rational functions.
\begin{qn}
  Let $P_1,P_2,Q_1,Q_1\in\mathbf{Z}[y]$ with $P_1(0)=P_2(0)=0$ be such that $P_1/Q_1$ and
  $P_2/Q_2$ are linearly independent rational functions. Does there exist a
  $\gamma=\gamma_{P_1,P_2,Q_1,Q_2}>0$ such that if $A\subset\mathbf{F}_p$ contains no nontrivial
  progressions
  \[
    x,x+\frac{P_1(y)}{Q_1(y)},x+\frac{P_2(y)}{Q_2(y)}
  \]
  then
  \[
    |A|\ll p^{1-\gamma}?
  \]
\end{qn}
This question is still open, and it is not clear how to adapt the arguments
of~\cite{BourgainChang17},~\cite{DongLiSawin20},~\cite{Peluse18},  or~\cite{Peluse19} to
attack it.

Based on the success story of the degree-lowering method, a reasonable roadmap to a fully
general quantitative polynomial Szemer\'edi theorem may be to first try to prove such a
result in finite fields, and then try to adapt the proof to the integer setting. But, I
think that even the finite field setting is likely to be very difficult, and
one should instead start with the function field setting, where the inverse theorems for
the $U^s$-norms produce genuine polynomial phases instead of nilsequences. Thus, I
propose the following problem.
\begin{problem}
  Prove reasonable bounds in the polynomial Szemer\'edi theorem in function fields.
\end{problem}
All that is known towards this problem is the work of Le--Liu~\cite{LeLiu2013}, Green~\cite{Green17}, and
Li--Sauermann~\cite{LiSauermann22} mentioned in Subsection~\ref{ssec:further}.

The main result of~\cite{BergelsonLeibman96} is actually a joint
multidimensional-polynomial generalization of Szemer\'edi's theorem.
\begin{theorem}\label{thm:multidimpoly}
  Let $v_1,\dots,v_m\in\mathbf{Z}^d$ be nonzero vectors and
  $P_1,\dots,P_m\in\mathbf{Z}[y]$ be polynomials with zero constant term. If
  $A\subset[N]^d$ contains no nontrivial multidimensional polynomial progressions
  \[
    x,x+P_1(y)v_1,\dots,x+P_m(y)v_m,
  \]
  then $|A|=o_{P_1,\dots,P_m,v_1,\dots,v_m}(N^d)$.
\end{theorem}
Thus, Problem~\ref{prob:quant} is the combination of two special cases of the following,
even more difficult, problem.
\begin{problem}
  Prove a quantitative version of Theorem~\ref{thm:multidimpoly} with reasonable bounds.  
\end{problem}
Not a single nonlinear, genuinely multidimensional (i.e., with $v_1,\dots,v_m$ not all contained in
a line) case of this problem is known in the integer setting. Until very recently, none
were known in the finite field setting either. Adapting the methods of Dong, Li, and Sawin~\cite{DongLiSawin20},
Han, Lacey, and Yang~\cite{HanLaceyYang21}  proved a power-saving bound for subsets of
$\mathbf{F}_p^2$ lacking nontrivial instances of the two-dimensional ``polynomial corner''
\[
(x,y),(x,y+P_1(z)),(x+P_2(z),y),
\]
whenever $P_1$ and $P_2$ have distinct degrees. Using the degree-lowering method,
Kuca~\cite{Kuca21} proved power-saving bounds for subsets of $\mathbf{F}_p^d$ lacking
polynomial corners
\begin{equation*}
  (x_1,\dots,x_d),(x_1+P_1(y),\dots,x_d),\dots,(x_1,\dots,x_d+P_d(y)),
\end{equation*}
of arbitrarily high dimension, again provided that $P_1,\dots,P_d$ have distinct
degrees. Replacing the distinct degree condition on the polynomials with a linear
independence condition in these results turned out to be a significant challenge, and was
only done very recently by Kuca~\cite{Kuca23}. Perhaps the next most attackable problem is
to find a joint generalization of Prendiville's work on arithmetic progressions with
perfect power common difference and Shkredov's result on corners, in both the integer and
finite field settings.
\begin{problem}
  Prove, in both $\mathbf{F}_p$ and $[N]$, reasonable quantitative bounds for sets lacking
  the multidimensional-polynomial configurations
  \begin{equation*}
    (x,y),(x,y+z^d),(x+z^d,y).
  \end{equation*}
\end{problem}

\section{Additional topics}\label{sec:additional}

\subsection{Comparing different notions of rank}

Following the emergence of the slice rank method, there has been a great deal of interest
in trying to find quantitative relations between various definitions of rank. One reason
for this is that, so far, the notion of slice rank does not seem sufficient to prove
power-saving bounds for sets lacking four-term arithmetic progressions in the finite field
model setting. It could be easier to show that the diagonal tensor of a set free of
four-term arithmetic progressions has small rank for some other definition of rank
different from slice rank. If we knew that this other notion of rank was roughly
quantitatively equivalent to slice rank, then we could obtain bounds for the size of sets
lacking four-term arithmetic progressions. Another reason for understanding the relations
between different definitions of rank is that many arguments in additive combinatorics,
particularly in the finite field model setting, split into a high rank case (the
``pseudorandom'' case) and a low rank case (the ``structured'' case), and proving stronger
quantitative relationships between analytic rank and other notions of rank can improve
these arguments.

We will now introduce the notion of analytic rank. In connection with the inverse theorem for the Gowers uniformity norms for polynomial
phases in the finite field model setting, Green and Tao~\cite{GreenTao09II} studied the
distribution of values of polynomials over finite fields. They showed that if the values
of a polynomial $P:\mathbf{F}_p^n\to\mathbf{F}_p$ are far from uniformly distributed in
$\mathbf{F}_p$ as the input ranges over $\mathbf{F}_p^n$, then $P$ can be written as a
low-complexity combination of polynomials of strictly smaller degree. More precisely, the
failure of equidistribution of $P$ can be measured by the correlation of $P$ with any
nontrivial additive character of $\mathbf{F}_p$: this is called the \defword{bias} of $P$,
\[
  \bias(P):=\left|\mathbf{E}_{\mathbf{x}\in\mathbf{F}_p^n}e_p(P(\mathbf{x}))\right|.
\]
Gowers and Wolf~\cite{GowersWolf11II} defined the \defword{analytic rank} of a polynomial
with nonzero bias to be $-1$ times the logarithm of its bias:
\[
  \arank(P):=-\log_p\left(\bias(P)\right).
\]
Note that, since $|\bias(P)|\leq 1$, the analytic rank of a nonconstant polynomial is
always positive.

The extent to which $P$ can be written as a low-complexity combination of polynomials of
strictly smaller degree is measured by its \defword{rank}, which is the smallest integer
$m$ for which there exist polynomials $P_1,\dots,P_m:\mathbf{F}_p^n\to\mathbf{F}_p$ of
degree strictly less than $\deg{P}$ and a function $Q:\mathbf{F}_p^m\to\mathbf{F}_p$ such
that $P=Q(P_1,\dots,P_m)$. Green and Tao~\cite{GreenTao09II} proved the following inverse
theorem for biased polynomials.
\begin{theorem}\label{thm:biasrank}
  Let $P:\mathbf{F}_p^n\to\mathbf{F}_p$ be a polynomial of degree strictly less than
  $p$. For all $\delta>0$, there exists an $R>0$ depending only on $\delta$ and $p$
  such that if $\bias(P)\geq\delta$, then $\rank(P)\leq R$.
\end{theorem}
Kaufman and Lovett~\cite{KaufmanLovett08} removed the high characteristic assumption from
Theorem~\ref{thm:biasrank}. Both of these arguments produce an upper bound for $R$ that
has Ackermann-type dependence on $1/\delta$ and $p$. This was improved to tower-type
dependence for fixed $p$ by Janzer~\cite{Janzer18}, and then to polynomial dependence by
Janzer~\cite{Janzer20} and Mili\'cevi\'c~\cite{Milicevic19}. Theorem~\ref{thm:biasrank}
and its quantitative improvements says that the rank of $P$ can be bounded in terms of the
analytic rank of $P$.

Another useful notion of rank is partition rank, which was first defined by
Naslund~\cite{Naslund20} in work extending the slice rank method.
\begin{Def}
  Let $M:(\mathbf{F}_P^n)^m\to\mathbf{F}_p$ be a multilinear form.
  \begin{enumerate}
  \item We say that $M$ has \defword{partition rank $1$} if there is a nonempty subset of indices
    $I\subset [m]$ and multilinear forms $M':(\mathbf{F}_p^n)^{|I|}\to\mathbf{F}_p$ and
    $M'':(\mathbf{F}_p^n)^{m-|I|}\to\mathbf{F}_p$ such that
    \[
      M(x_1,\dots,x_m):=M'(x_i:i\in I)M''(x_j:j\in[m]\setminus I)
    \]
    for all $x=(x_1,\dots,x_m)\in(\mathbf{F}_p^n)^m$.
  \item The \defword{partition rank} of $M$, denoted by $\prank{M}$, is the smallest nonnegative
    integer $k$ such that $M$ can be written as the sum of $k$ multilinear forms of
    partition rank $1$.
  \end{enumerate}
\end{Def}
In~\cite{Naslund20b}, Naslund used a variant of the slice rank method adapted to
partition rank to obtain the first power-saving bounds for the Erd\H{o}s--Ginzberg--Ziv
constant for high-dimensional vector spaces over finite fields.

Lovett~\cite{Lovett19} and, independently, Kazhdan and Ziegler~\cite{KazhdanZiegler2018}
showed that the analytic rank of a multilinear form is bounded above by its partition rank
and slice rank. Janzer~\cite{Janzer20} and Mili\'cevi\'c~\cite{Milicevic19} showed that,
in the reverse direction, the partition rank of a multilinear form can be bounded in terms
of its analytic rank, with polynomial dependence. Recently, Moshkovitz and
Zhu~\cite{MoshkovitzZhu22} proved almost-linear dependence of partition rank on analytic
rank: $\prank(M)\ll_m\arank(M)(\log(1+\arank(M)))^{m-1}$.

\subsection{The multidimensional Szemer\'edi theorem}
Reasonable bounds in Theorem~\ref{thm:multidim} are only known for only one genuinely
multidimensional configuration: two-dimensional corners, due to work of
Shkredov~\cite{Shkredov06I,Shkredov06II}, who showed that if $A\subset[N]^2$ contains no
nontrivial corners, then
\begin{equation*}
|A|=O\left(\frac{N^2}{(\log\log{N})^c}\right)  
\end{equation*}
for some $c>0$. Shkredov proved this result in 2006, and more than fifteen years later, no
reasonable bounds are known for sets lacking any multidimensional four-point
configuration. I have, however, managed to make a small amount of progress in the finite
field model setting~\cite{Peluse24}.
\begin{theorem}\label{thm:Lshapes}
  There exists an absolute constant $m\in\mathbf{N}$ such that the following holds. Fix
  $p\geq 11$. If $A\subset(\mathbf{F}_p^n)^2$ contains no nontrivial L-shapes,
  \begin{equation*}
    (x,y),(x,y+z),(x,y+2z),(x+z,y)\qquad(z\neq 0),
  \end{equation*}
  then
  \begin{equation*}
    |A|\ll\frac{p^{2n}}{\log_{(m)}p^{n}}.
  \end{equation*}
\end{theorem}
The size of $m$ in this result depends on the number of iterated exponentials appearing in
the Gowers--Mili\'cevi\'c inverse theorem for the $U^{10}$-norm, and can be taken to be
$24$ trillion. Any improvement in the quantitative aspect of Gowers and Mili\'cevi\'c's
theorem would automatically improve the size of $m$. The argument used to prove this
result is robust enough that it adapts to the integer setting, though with significant
technical difficulties arising from the need to work relative to Bohr sets.

Proving reasonable bounds for any other genuinely multidimensional four-point
configuration that is not a linear image of an L-shape seems to be an extremely difficult
problem, even in the finite field model setting. Such configurations include axis-aligned
squares and three-dimensional corners.
\begin{problem}\label{prob:squares}
  Fix a sufficiently large prime $p$. Prove reasonable bounds
  for subsets of $(\mathbf{F}_p^n)^2$ containing no nontrivial axis-aligned squares,
  \[
    (x,y),(x,y+z),(x+z,y),(x+z,y+z)\qquad(z\neq 0).
  \]
\end{problem}
The count of axis-aligned squares in a set is controlled by the following two-dimensional
analogue of a Gowers uniformity norm:
\begin{equation*}
  \|f\|_{\bullet}:= \left(\mathbf{E}_{x,y,h_1,h_2,h_3\in\mathbf{F}_5^n}\Delta_{(h_1,0),(0,h_2),(h_3,h_3)}f(x,y)\right)^{1/8}.
\end{equation*}
So far, no one has been able to prove an inverse theorem for $\|\cdot\|_{\bullet}$.
\begin{problem}
  Prove an inverse theorem, preferably with reasonable bounds, for the norm $\|\cdot\|_{\bullet}$.
\end{problem}
This norm also shows up in the analysis of three-dimensional corners.
\begin{problem}
  Fix a sufficiently large prime $p$. Prove reasonable bounds for subsets of
  $(\mathbf{F}_p^n)^3$ containing no nontrivial three-dimensional corners,
  \[
    (x,y,z),(x,y,z+w),(x,y+w,z),(x+w,y,z).
  \]
\end{problem}
By a projection argument, any solution to this problem would yield a solution to
Problem~\ref{prob:squares}, but this problem is strictly harder since one must also find a
way to get around using Szemer\'edi's regularity lemma (or related results with very weak
bounds) when ``pseudorandomizing'' the objects produced by various inverse theorems.

\subsection{True complexity}

The \textit{true complexity} of a system of linear forms
$\phi_1,\dots,\phi_m\in\mathbf{Z}[x_1,\dots,x_d]$ in $d$ variables is the smallest natural
number $s$ such that the following statement holds: for every $\varepsilon>0$, there
exists a $\delta>0$ such that if $p\gg_{\phi_1,\dots,\phi_m}1$ and
$f:\mathbf{F}_p^n\to\mathbf{C}$ is a $1$-bounded function on $G:=\mathbf{F}_p^n$
satisfying $\|f\|_{U^{s+1}}<\delta$, then
\[
  \left|\mathbf{E}_{\mathbf{x}\in G^d}\prod_{i=1}^mf(\phi_i(\mathbf{x}))\right|<\varepsilon.
\]
The notion of true complexity, defined by Gowers and Wolf in~\cite{GowersWolf10}, is
important both in additive combinatorics and in ergodic theory. Gowers and
Wolf conjectured that the true complexity of a system of linear
equations is equal to the smallest integer $s$ such that the polynomials
\begin{equation*}
\phi_1(\mathbf{x})^{s+1},\dots,\phi_m(\mathbf{x})^{s+1}  
\end{equation*}
are linearly independent over $\mathbf{Q}$, and proved their conjecture in the finite
field model setting~\cite{GowersWolf11I,GowersWolf11II} (i.e., in the regime where $p$ is
fixed and $n\to\infty$) in full generality and in the setting of cyclic
groups~\cite{GowersWolf11III} (i.e., in the regime where $n=1$ and $p\to\infty$) when
$s=2$. Green and Tao~\cite{GreenTao10II} proved the conjecture of Gowers and Wolf in
cyclic groups for systems of linear forms satisfying a technical condition known as the
``flag condition'', and Altman~\cite{Altman22} recently proved the conjecture in full
generality in the cyclic group setting.

All of the aforementioned proofs use the inverse theorems for the Gowers uniformity norms
and some variant of a regularity lemma, and thus mostly produce very poor dependence of
$\delta$ on $\varepsilon$. In a pair of amazing papers,
Manners~\cite{Manners18II,Manners21} has given proofs of the conjecture of Gowers and Wolf
in both the finite field model and integer settings using only repeated applications of
the Cauchy--Schwarz inequality (along with essentially trivial moves like changes of
variables and applying the pigeonhole principle), which produces polynomial dependence of
$\delta$ on $\varepsilon$ (though the polynomial dependence, necessarily, depends on the
system of linear forms).

One can also define the true complexity of a polynomial progressions. Let
$P_1,\dots,P_m\in\mathbf{Z}[y]$ be a collection of polynomials and $0\leq i\leq m$ be a
natural number. We say that the polynomial progression $x,x+P_1(y),\dots,x+P_m(y)$ has
\defword{true complexity $s$ at $i$} if $s$ is the smallest natural
number for which the following statement holds: for every $\varepsilon>0$, there
exists a $\delta>0$ such that if $p\gg_{P_1,\dots,P_m}1$ and
$f_0,\dots,f_m:\mathbf{F}_p\to\mathbf{C}$ are $1$-bounded functions such that $\|f_i\|_{U^{s+1}}<\delta$, then
\[
  \left|\mathbf{E}_{x,y\in\mathbf{F}_p}f_0(x)\prod_{j=1}^mf_j(x+P_j(y))\right|<\varepsilon.
\]
A polynomial progression of length $m+1$ has true complexity $s$ if its maximum true
complexity at $i$ over $0\leq i\leq m$ is $s$.  Very little is known about the true
complexity of a general polynomial progression. Bergelson, Leibman, and
Lesigne~\cite{BergelsonLeibmanLesigne07} asked whether any polynomial progression of
length at most $m+1$ must have true complexity at most $m-1$, and this question is still
entirely open for $m>3$ (Frantzikinakis~\cite{Frantzikinakis08} verified it when
$m\leq 3$). In contrast, a simple modification of the proof of Lemma~\ref{lem:gvnAP} shows
that the true complexity of an arbitrary system of $m+1$ linear forms of finite complexity
has true complexity at most $m-1$.

The main difficulty in extending Frantzikinakis's argument to longer progressions lies in
analyzing the distribution of certain polynomial sequences on nilmanifolds. Thus, since
the inverse theorems for the $U^s$-norms in the finite field model setting give
correlation with a genuine polynomial phase, this problem may be more approachable in the
function field setting.
\begin{problem}
  Resolve the question of Bergelson, Leibman, and Lesigne in the function field setting.
\end{problem}

In~\cite{Kuca23II}, Kuca formulated a polynomial progression analogue of the conjecture of
Gowers and Wolf, and verified some special cases of his conjecture.
\begin{conj}
Let
$P_1,\dots,P_m\in\mathbf{Z}[y]$ be a collection of polynomials and $0\leq i\leq m$ be a
natural number. The polynomial progression $x,x+P_1(y),\dots,x+P_m(y)$ has true complexity
$s$ at $i$ if $s$ is the smallest natural number such that for any algebraic relation
\[
Q_0(x)+Q_1(x+P_1(y))+\dots+Q_m(x+P_m(y))=0
\]
with $Q_0,\dots,Q_m\in\mathbf{Z}[z]$ satisfied by $x,x+P_1(y),\dots,x+P_m(y)$, the degree
of $Q_i$ is at most $s$.
\end{conj}
Similarly to the question of Bergelson, Leibman, and Lesigne, Kuca's conjecture may be
more attackable in the function field setting.
\begin{problem}
  Prove Kuca's conjecture in the function field setting.
\end{problem}

Finally, it may be possible to prove true complexity statements for polynomial
progressions just using repeated applications of the Cauchy--Schwarz inequality (along
with trivial moves). Manners has posed the following concrete problem.
\begin{problem}\label{prob:zero}
  Prove that the nonlinear Roth configuration has true complexity zero using only the
  Cauchy--Schwarz inequality.
\end{problem}
In Subsection~\ref{ssec:deglow}, we proved that the nonlinear Roth configuration has true
complexity zero by combining many applications of the Cauchy--Schwarz inequality and the
pigeonhole principle with some basic Fourier analysis.


\thankyou{I would like to thank Thomas Bloom and Julia Wolf for helpful comments on
  earlier drafts and Dan Altman for helpful discussions. I was supported by the NSF
  Mathematical Sciences Postdoctoral Research Fellowship Program under Grant
  No. DMS-1903038 while writing the bulk of this survey.}

 \bibliographystyle{amsplain}
 \bibliography{bib.bib}






\bigskip

\myaddress


\end{document}